\documentclass[12pt, reqno]{amsart}
\allowdisplaybreaks
\usepackage{amssymb,amsthm,amsmath,bm}
\usepackage[numbers,sort&compress]{natbib}
\usepackage{color}
\usepackage{enumerate}

\title[]{Global well-posedness and large time behavior of 3D incompressible inhomogeneous magnetohydrodynamic equations in the exterior of a cylinder}
\author{Jitao Liu,\,Min Liu$^*$}
\address[Jitao Liu]{Department of Mathematics, School of Mathematics, Statistics and Mechanics, Beijing University of Technology, Beijing, 100124, P. R. China.}
\email{jtliu@bjut.edu.cn,\,\,\,jtliumath@qq.com}
\address[Min Liu]{College of Mathematics and Information Science, Zhengzhou University of Light Industry, Zhengzhou, Henan, 450002,  P. R. China}
\email{mliu@zzuli.edu.cn,\,\,\,liumin4521@163.com}


\theoremstyle{plain}

\newtheorem{theorem}{Theorem}[section]
\newtheorem*{theorem*}{Theorem}
\newtheorem{lemma}{Lemma}[section]

\theoremstyle{definition}

\newtheorem{remark}{Remark}[section]

\newcommand{\bga}{\begin{array}}
\newcommand{\eda}{\end{array}}
\newcommand{\p}{\partial}

\let\f=\frac
\let\p=\partial
\def\R{\Bbb R}

\let\f=\frac
\let\p=\partial
\def\R{\Bbb R}

\newcommand{\beq}{\begin{equation}}
\newcommand{\eeq}{\end{equation}}
\newcommand{\ben}{\begin{eqnarray}}
\newcommand{\een}{\end{eqnarray}}
\newcommand{\beno}{\begin{eqnarray*}}
\newcommand{\eeno}{\end{eqnarray*}}

\begin{document}


\begin{abstract}
{When the vaccum is allowed, if the global existence and uniqueness of strong solutions to three dimensional incompressible inhomogeneous magnetohydrodynamic equations holds true or not has always been a challenging open problem, even for the magnetofluids with special structures. In this paper, through deeply exploring the internal structure and characteristic of axisymmetric flows, we obtain some new discoveries and give a partial answer to above issue. More precisely, we prove that the axisymmetric magnetofluids flowing in the exterior of a cylinder will definitely admits a unique strong solution that exists globally in time without any compatibility conditions and small assumptions imposed on the initial data. Furthermore, we establish the algebraic decay rates for the time and spatial derivatives of both velocity and magnetic fields. To the best of our knowledge, this result gives the first unique 3D large solution existing globally in time.}
\end{abstract}
\footnote{*Corresponding author}
\maketitle
\tableofcontents

\vskip .3in
\section{Introduction and main results}

The incompressible magnetohydrodynamics ($\it MHD\,\,for\,\,short$) equations describe the motion of electrically conducting fluids, that is the dynamic motion of fluid and magnetic field interact strongly with each other, the readers can see \cite{Dav-2001} for details. In particular,
the three dimensional ($\it 3D\,\,for\,\,short$) incompressible inhomogeneous MHD equations are written as,
\begin{equation}\label{1-1}
\left\{\begin{array}{ll} \rho_t +\text{div}(\rho u)=0,\quad&\text{in}\,[0, T)\times\Omega{\,},
\\\noalign {\vskip 4pt}
(\rho u)_t+\text{div}(\rho u\otimes u)-\mu\triangle u+\nabla p=b\cdot\nabla b,\quad&\text{in}\,[0, T)\times\Omega{\,},
\\\noalign {\vskip 4pt}
b_t-\nu\triangle b+u\cdot\nabla b-b\cdot\nabla u=0,\quad&\text{in}\,[0, T)\times\Omega{\,},
\\\noalign {\vskip 4pt}
\text{div}\, u=\text{div}\, b=0, \quad&\text{in}\,[0, T)\times\Omega,
\end{array}\right.
\end{equation}
where $\rho=\rho(x, t)$ is the density, $u=(u^1, u^2, u^3)$ and $b=(b^1, b^2, b^3)$ represent the velocity and magnetic field respectively, $p=p(x, t)$ denotes the pressure of fluid. The nonegative constants $\mu\geq0$ and $\nu\geq0$ stand for the viscosity and resistivity coefficients separately, $\Omega\in\mathbb{R}^3$ is a domain. Without loss of generality, we assume $\mu=\nu=1$ in this paper and the initial data are given by
\begin{align}\label{1-2}
(\rho, \rho u, b)|_{t=0}=(\rho_0, \rho_0u_0, b_0).
\end{align}

If the motion occurs without magnetic field ($\mathrm{i.\,e.}\,b=0$), the system \eqref{1-1} reduces to the classical incompressible inhomogeneous Navier-Stokes equations. For this model in the whole space, Antontsev-Kazhikov \cite{b1.2} first established the global existence of weak solutions without vacuum (see also \cite{b1.3, b1.4}).If the vacuum is allowed, Simon \cite{b1.5, b1.6} and Lions \cite{b1} proved the global existence of
weak solutions. For the bounded domain $\Omega$ with Dirichlet boundary condition, Ladyzhenskaya-Solonnikov \cite{b1.7} first obtained the global well-posedness of strong solutions. In 2003, Choe-Kim \cite{b1.8} proved the local existence and uniqueness of strong solutions for 3D bounded and unbounded domains $\Omega$ if the initial data satisfies the following compatibility condition
\begin{align}\label{1-5}
-\mu\triangle u_0+\nabla p_0=\rho_0^{\frac{1}{2}}g,\quad\mathrm{for}\,\,\mathrm{some}\,\, (p_0, g)\in H^1(\Omega)\times L^2(\Omega).
\end{align}
Recently, this compatibility condition \eqref{1-5} was removed by Li in \cite{b1.12}. In the case of global well-posedness theory of strong solutions,
for 2D model, it was solved by Huang-Wang \cite {HW-2014} in 2014 and L\"{u}-Shi-Zhong \cite{b1.11} in 2018 successively. For $\Omega=\mathbb{R}^3$ or be a bounded domain and the vacuum is allowed, Craig-Huang-Wang \cite{b1.13} proved the  global existence and uniqueness of strong solutions supposing  $\|u_0\|_{\dot{H}^{\f12}(\Omega)}$ is small (see also a newer paper \cite{b1.14}). For the axisymmetric flows without swirl, Abidi-Zhang \cite{b1.24} established the global well-posedness of strong solutions provided that $\|\frac{\rho_0^{-1}-1}{r}\|_{L^{\infty}(\mathbb{R}^3)}$ is sufficiently small. This result was then extended to the case with swirl for sufficient small  $\|\frac{\rho_0^{-1}-1}{r}\|_{L^{\infty}(\mathbb{R}^3)}$ and $\|u_0^{\theta}\|_{L^3(\mathbb{R}^3)}$ by Chen-Fang-Zhang in \cite{b1.25}. It should be noted that for 3D model, the small assumptions on the initial data were essential in above results. Until recently, Guo-Wang-Xie \cite{GWX-2021} established the global existence and uniqueness of axisymmetric strong solutions in the exterior of a cylinder without any small assumptions on the initial data. Afterwards, Wang-Guo \cite{WG-2025} and Liu \cite{MinL2025} studied the large time behavior of this strong solution and obtained its algebraic and exponential decay rates of velocity field respectively.

 When the magnetic field is taken into account, there will be more and stronger nonlinear coupling effect and the situation becomes more complicated and quite different from incompressible inhomogeneous Navier-Stokes equations. For 2D bounded domain and the vacuum is allowed, Huang-Wang \cite{HW-2013} established the global existence and uniqueness of strong solutions provided that the compatibility condition holds. This compatibility condition was removed by Zhong in \cite{Zhongx2021} later. In 2014, for 2D Cauchy problem, Gui \cite{Gui-2014} proved that the system is globally well-posed for a generic family of variations of initial data and an inhomogeneous electrical conductivity. For 3D whole space,  Abidi-Paicu \cite{AP-2008} first established the global existence and uniqueness of strong solutions with small initial data in critical Besov spaces. For 3D periodic domain or bounded domain and the vacuum is allowed, Xu et al. \cite{XZQF-2022} proved the global well-posedness of strong solutions with some smallness assumptions on the initial data. Moreover, for the axisymmetric MHD flows without swirl and with only swirl component of magnetic field, Liu \cite{LWJ-2024} obtained the global well-posedness provided that $\|\frac{\rho_0^{-1}-1}{r}\|_{L^{\infty}(\mathbb{R}^3)}$ is sufficiently small, see also \cite{JL-2015} for the homogeneous model.

 From the literature mentioned above, we can discover that for 3D model of system \eqref{1-1}, to establish the unique strong solution existing globally in time, the small assumptions on the initial data are essential. There is not any unique 3D large solution existing globally in time still, even for axisymmetric flows. With this open issue in mind and motivated by the recent work \cite{GWX-2021}, in this paper, we make the first attempt to look for the unique and strong large axisymmetric solution in the exterior of a cylinder
 \begin{align}\label{domain}
 	\Omega=\{(x_1, x_2, x_3)\in\mathbb{R}^3: r^2=x_1^2+x_2^2>1, x_3\in\mathbb{R}\},
 \end{align}
 with Dirichlet boundary conditions,
 \begin{align}\label{1-2boundary}
 u=0,\,\,b=0\,\,\text{on}\,\,[0, T)\times\partial\Omega.
 \end{align}
Through fully exploring the internal structure and characteristic of axisymmetric flows and the domain \eqref{domain} considered in current paper, we obtain some new observations (the details will be given after Theorem \ref{theorem.1}). With the help of them, we are able to establish the global existence, uniqueness and large time decay rates of axisymmetric strong solutions, which are summarized in the following main theorem.

\begin{theorem}\label{theorem.1}
Let $\Omega$ be the exterior of a cylinder, for given $\bar{\rho}>0$, assume that the initial data $(\rho_0, u_0, b_0)$ is axisymmetric and satisfies
\begin{align}\label{1-7}
	0\leq\rho_0\leq\bar{\rho},\,\,\rho_0-\bar{\rho}\in L^{\frac{3}{2}}\cap\dot{H}^1(\Omega),\,\,u_0\in H^1_{0, \sigma}(\Omega),\,\,b_0\in H^1_{0, \sigma}(\Omega).
\end{align}
Then for any given $0<T<\infty$ and $q\in[2, \infty)$, there is a unique global strong solution $(\rho, u, b)$ of the system \eqref{1-1}-\eqref{1-2} and \eqref{1-2boundary} such that 
\begin{align}\label{1-9}
	\left\{
	\begin{aligned}
		&0\leq\rho-\bar{\rho}\in L^{\infty}([0, +\infty); L^{\f32}\cap L^{\infty}\cap \dot{H}^1(\Omega))\cap C([0, +\infty);L^{q}(\Omega)),\\
		&\rho u\in C([0, +\infty); L^2(\Omega)),\,\,\rho_t\in L^4([0, +\infty); L^2(\Omega)),\,\,\sqrt{\rho}u_t\in L^2([0, +\infty); L^2(\Omega)),\\
		&u\in L^\infty([0, +\infty); H^1_{0,\sigma}(\Omega))\cap L^2([0, +\infty); {H}^2(\Omega)),\,\,\sqrt{t}u_t\in L^2([0, +\infty); H^1(\Omega)),\\
		&\sqrt{t} \nabla u\in L^\infty([0,+\infty); L^2(\Omega))\cap L^2([0,+\infty); \dot{H}^1(\Omega)),\\
		&{t}\nabla^2u,{t}\nabla^2b\in L^\infty([0,+\infty); L^2(\Omega))\cap L^2([0, +\infty); L^6(\Omega)),\\
		&b\in L^\infty([0,+\infty); L^4\cap H^1_{0,\sigma}(\Omega))\cap L^2([0,+\infty); {H}^2(\Omega)),\,\,b_t\in L^2([0,+\infty); L^2(\Omega)),\\
		&\sqrt{t} \nabla b\in L^\infty([0,+\infty); L^2(\Omega))\cap L^2([0,+\infty); \dot{H}^1(\Omega)),\,\,\sqrt{t}b_t\in L^2([0, +\infty); H^1(\Omega)).
	\end{aligned}
	\right.
\end{align}
Moreover, for any $t>0$, $(\rho, u, b)$ has the following time-asymptotically decay rates:
\begin{align}\label{1-10}
\begin{aligned}
&\|\nabla u(\cdot, t)\|_{L^2(\Omega)}+\|\nabla b(\cdot, t)\|_{L^2(\Omega)}\leq C(1+t)^{-\frac{1}{2}},\\
&\|\nabla^2 u(\cdot, t)\|_{L^2(\Omega)}+\|\nabla^2 b(\cdot, t)\|_{L^2(\Omega)}\leq Ct^{-1},\\
&\|\sqrt{\rho}u_t(\cdot, t)\|_{L^2(\Omega)}+\|b_t(\cdot, t)\|_{L^2(\Omega)}\leq Ct^{-1},
\end{aligned}
\end{align}
where the genuine constant $C$ depends only on $\bar{\rho}$, $\|\rho_0-\bar{\rho}\|_{L^{\f32}(\Omega)}$,  $\|\nabla {\rho}_0\|_{L^2(\Omega)}$, $\|u_0\|_{H^1(\Omega)}$ and $\|b_0\|_{H^1(\Omega)}$.
\end{theorem}

\begin{remark}
	Theorem \ref{theorem.1} provides the first unique large solution for 3D incompressible inhomogeneous MHD equations, even for the magnetofluids with special structures. If $\rho\equiv 1$, Theorem \ref{theorem.1} implies the global well-posedness result for the corresponding 3D homogeneous model.
\end{remark}

\begin{remark}
	If the initial density has compact support, then \eqref{1-7} for density is satisfied naturally, so
	Theorem \ref{theorem.1} still holds under this case.
\end{remark}

\begin{remark}
Theorem \ref{theorem.1} generalizes the result in \cite{GWX-2021} for 3D incompressible inhomogeneous Navier-Stokes equations in four aspects by assuming $b=0$. {\it The first}, we remove the compatibility conditions imposed on the initial data there. {\it The second}, the initial assumptions \eqref{1-7} are much weaker than \cite{GWX-2021}, where $\rho_0-\bar{\rho}\in L^{\frac{3}{2}}\cap{H}^2(\Omega)$ and $u_0\in H^1_{0, \sigma}\cap H^2(\Omega)$ are required there. {\it The third}, due to the a priori estimates obtained here is independent of time $T$ (see Lemmas \ref{lem.B3.01}-\ref{lem.B3.4} for detail), the existence time of strong solution in Theorem \ref{theorem.1} can reach infinity, rather than any fixed $T$ in \cite{GWX-2021}. {\it The fourth}, compared with \cite{GWX-2021}, we further establish the algebraic decay rates for the time and spatial derivatives of both velocity and magnetic fields. If without magnetic field, the decay rates of velocity field can be updated to exponential, the readers can refer to \cite{MinL2025} by the second author of this paper for details.
\end{remark}

\begin{remark}
	When the axis is not included in the domain, the corresponding issue inherits more features of 2D flows. Some key estimates, such as Lemmas \ref{lem.B30} and \ref{newgn} depend on this property crucially,  therefore it is hard to extend this result to 3D whole space or any domain including the axis directly.
\end{remark}

\noindent{\bf The main idea:} In the process of solving this issue, there are two main challenges appear. {\it The first one} is how to overcome the degeneracy caused by the vacuum, which are summarized as follows.
\begin{enumerate}[\quad 1.]
	\item In the literature \cite{GWX-2021}, for the axisymmetric velocity filed flowing in the exterior of a cylinder, the authors established the following new inequality
	\begin{align}\label{key1}
		&\int_s^T\|\nabla u\|_{L^\infty}^2\,d\tau\notag\\
		\leq& C\|\nabla u\|^2_{L^2([s,T];L^2)}\ln\left[e+\|\nabla u\|^2_{L^\infty([0,T];L^2)}+\|\sqrt{\rho} u_t\|^2_{L^2([0,T];L^2)}\right]+C,
	\end{align}
and then use \eqref{key1} to get the key $L^\infty([0,T];L^2)$ estimates of $\nabla u$.
	\item However, the method taken in \cite{GWX-2021} and \eqref{key1} does not work for our issue. This is because for us, their approach will make troubles in removing the compatibility conditions and establishing the algebraic decay rates. To get over this obstacle, we establish the new inequality
	\begin{align}\label{key2}
		\|\sqrt{\rho}u\|^2_{L^4(\Omega)}
		\leq C(\bar{\rho})\left(1+\|\sqrt{\rho}u\|_{L^2(\Omega)}\right)\|u\|_{H^1(\Omega)}\sqrt{\ln\left(2+\|u\|^2_{H^1(\Omega)}\right)},
	\end{align}
	which only holds for 2D flows before. With the help of \eqref{key2}, the trouble caused by the vacuum can be solved.
\end{enumerate}

{\it The second one} is the new nonlinear coupling terms involving the velocity and magnetic fields. Evidently, \eqref{key2} does not work for these terms.
Because the exterior of a cylinder is a 3D model with non-compact boundary, it is not the exterior domain. From here, according to classical Gagliardo-Nirenberg inequalities, there holds
\begin{align}\label{key10}
	\|\nabla^jf\|_{L^p(\Omega)}\leq \tilde{C_1}\|\nabla^mf\|_{L^r(\Omega)}^\alpha\|f\|_{L^q(\Omega)}^{1-\alpha}+\tilde{C_2}\|f\|_{L^s(\Omega)},
\end{align}
where the index $\alpha$ is as same as ordinary 3D flows. Trivially, \eqref{key10} can not help us establishing the unique global strong solution. However, we notice that the components of axisymmetric flows only depend on variables $r$ and $z$ and $r\geq 1$ always holds true in $\Omega$. Under this case, the integrands can be seen as the 2D functions in the exterior of a unit circle. Based on this observation, we can derive the following new Gagliardo-Nirenberg type inequalities holding for our flow, that is 
\begin{align}\label{key20}
	\|\nabla^ju\|_{L^p(\Omega)}\leq \tilde{C}\|\nabla^mu\|_{L^r(\Omega)}^\alpha\|u\|_{L^q(\Omega)}^{1-\alpha},
\end{align}
where 
\begin{align}\label{key30}
	\frac{1}{p}=\frac{j}{2}+\alpha\left(\frac{1}{r}-\frac{m}{2}\right)+\left(1-\alpha\right)\frac{1}{q}, \qquad \frac{j}{m} \leq \alpha \leq 1.
\end{align}
Compared with \eqref{key10}, there are two advantages in
\eqref{key20} and \eqref{key30}. The first and also most important is that the index $\alpha$ is the same as 2D flows, which provides us the foundation to deal with the new nonlinear terms and derive the global a priori estimates. The second is that the norm $\tilde{C_2}\|f\|_{L^s(\Omega)}$ disappears, that plays an important role is establishing the uniform estimates independent of $T$ and time-asymptotically algebraic decay rates for strong solutions. Thanks to the two key findings and very delicate a priori estimates, we can achieve our goal.

This paper is organized as follows. In section 2, we introduce some notations and technical lemmas used for the proof of main theorems. In section 3, we will concentrate on the proof of Theorem \ref{theorem.1}. Section 4 is devoted to the local well-posedness of strong solutions.

\section{Preliminary}
\par In this section, we first recall some well-known inequalities, the stokes estimates and then use them to present some new inequalities only for axisymmetric flows, which will play the key role in the subsequent proofs. To start with, we introduce the notations and conventions used throughout this paper. First of all, we take the notation
\begin{align}
\int f dx\triangleq\int\limits_{\Omega} f dx\nonumber,
\end{align}
for simplicity. For $1\leq p\leq\infty$ and $k\geq1$, we use $L^p=L^p(\Omega)$ and $W^{k, p}=W^{k, p}(\Omega)$ to denote the standard Sobolev spaces. When $p=2$, we also use the shorthand notations $H^k=W^{k, 2}(\Omega)$. Denote the closure of $C_0^{\infty}(\Omega)$ in $H^1(\Omega)$ by $H_0^1$ and $H_{0, \sigma}^1$ to be the closure of $C_{0, \sigma}^{\infty}(\Omega)=\{u\in C_0^{\infty}(\Omega): \mathrm{div}\,u=0,\,\,\mathrm{in}\,\,\Omega\}$ in $H^1(\Omega)$.

\subsection{Some well-known tools} The first lemma to give is the following classical Gagliardo-Nirenberg inequalities (see for example \cite{exinequality}).

\begin{lemma} \label{lem.B2}
Let $ \Omega \subset \R^n $ be a domain, $ 1 \leq p, q, r \leq \infty $, $ \alpha > 0 $ and $ j < m $ be nonnegative integers such that
\begin{align}\label{re0}
\frac{1}{p} - \frac{j}{n} = \alpha\,\left( \frac{1}{r} - \frac{m}{n} \right) + (1 - \alpha)\frac{1}{q}, \qquad \frac{j}{m} \leq \alpha \leq 1,
\end{align}
then every function $ f: \Omega \mapsto \mathbb{R} $ that lies in $ L^q(\Omega) $ with $m^{\rm th}$ derivative in
$ L^r(\Omega) $ also has $j^{\rm th}$ derivative in
$ L^p (\Omega) $. Furthermore, it holds that	
	\begin{align}\label{2-2}
		\|\nabla^jf\|_{L^p(\Omega)}\leq \tilde{C_1}\|\nabla^mf\|_{L^r(\Omega)}^\alpha\|f\|_{L^q(\Omega)}^{1-\alpha}+\tilde{C_2}\|f\|_{L^s(\Omega)},
	\end{align}
where $ s>0 $ is arbitrary and the constants $\tilde{C_1}$ and $\tilde{C_2}$ depend upon $\Omega$ and the indices $ n, m, j, q, r, s, \alpha $ only. Specifically, the constant $\tilde{C_2}$ can be equal to zero either if $u\in W_0^{m,p}(\Omega)$, or $\Omega=\mathbb{R}^n$ or $\Omega$ be an exterior domain.
\end{lemma}

Due to the region we consider here is a unbounded domain with Drichlet boundary condition, to derive the high order derivative estimates of velocity field, we need the following estimates on the Stokes equations in \cite{Cai-LY} (see also Theorem V.4.8 in \cite{b4}) with the constant $C$ independent of the area of domain.

\begin{lemma}\label{lem.B5}
	Let $\Omega$ be a domain of $\mathbb{R}^3$, whose boundary is uniformly of class $C^3$. Assume that $u\in H_{0, \sigma}^1(\Omega)$ is a weak solution to the following Stokes equations,
	\begin{align}\label{2.5-1}
		\left\{
		\begin{aligned}
			&-\triangle u+\nabla p=F,\quad &\mathrm{in}\,\,\Omega,\\
			&\mathrm{div}\,u=0,\quad &\mathrm{in}\,\,\Omega,\\
			&u=0,\quad &\mathrm{on}\,\,\partial\Omega.
		\end{aligned}
		\right.
	\end{align}
	Then for any $f\in L^p(\Omega)$ with $p\in(1, \infty)$, it holds that
	\begin{align}\label{2.5-1}
		\|\nabla^2u\|_{L^p(\Omega)}+\|\nabla p\|_{L^p(\Omega)}\leq C\|F\|_{L^p(\Omega)},
	\end{align}
	where the genuine constant $C$ depends only on $p$ and the $C^3$-regularity of $\partial\Omega$ (not on the size of $\partial\Omega$ or $\Omega$).
\end{lemma}

Finally, we state a Gr\"{o}nwall's type inequality (see \cite{b1.12} for details), which is used to prove the uniqueness of strong solutions and a classical Lemma originating from Desjardins in Lemma 1 of \cite{Des-1997}.

\begin{lemma}\label{lem.B6}
	Given a positive time T and nonnegative functions f, g, G on [0, T], with f and g being absolutely continuous on $[0,T]$. Suppose that
	\begin{align*}
		\left\{
		\begin{aligned}
			&\frac{d}{dt}\,f(t)\leq A\sqrt{G(t)},\\
			&\frac{d}{dt}\,g(t)+G(t)\leq\alpha(t)g(t)+\beta(t)f^2(t),\\
			&f(0)=0,
		\end{aligned}
		\right.
	\end{align*}
	a.e. on $(0, T)$, where $A$ is a positive constant, $\alpha$ and $\beta$ are two nonnegative functions satisfying
	\begin{align*}
		\alpha(t)\in L^1(0, T),\quad\mathrm{and}\quad t\beta(t)\in L^1(0, T).
	\end{align*}
	Then, the following estimates
	\begin{align*}
		f(t)\leq A\sqrt{g(0)}\sqrt{t}e^{\frac{1}{2}\int_0^t\left(\alpha(s)+A^2s\beta(s)\right)\,ds},
	\end{align*}
	and
	\begin{align*}
		g(t)+\int_0^tG(s)\,ds\leq g(0)e^{\int_0^t\left(\alpha(s)+A^2s\beta(s)\right)\,ds},
	\end{align*}
	 hold for $t\in[0, T]$, which, in particular, imply $f\equiv0$, $g\equiv0$ and $G\equiv0$ provided $g(0)=0$.
\end{lemma}

\begin{lemma}\label{lem.B30}
	Let $B=\{(x_1, x_2)\in\mathbb{R}^2: r^2=x_1^2+x_2^2\leq1\}$ be a unit circle and $B^{C}=\mathbb{R}^2-B$ be the exterior of $B$. Supposing $u\in H_0^1(B^{C})$, then there exists a genuine constant $C$ such that
	\begin{align}\label{2.3-10}
		\|\sqrt{\rho}u\|^2_{L^4(B^{C})}
		\leq C(\bar{\rho})\left(1+\|\sqrt{\rho}u\|_{L^2(B^{C})}\right)\|u\|_{H^1(B^{C})}\sqrt{\ln\left(2+\|u\|^2_{H^1(B^{C})}\right)}.
	\end{align}
\end{lemma}
\begin{proof} For the case of $\mathbb{R}^2$, the corresponding inequality has been proved in Lemma 1 of \cite{Des-1997}. Due to $u=0$ on $\partial B^{C}$, to prove \eqref{2.3-10}, it suffices to take zero extension of $u$ outside $B^{C}$.
\end{proof}

\subsection{New lemmas for axisymmetric flows}

Initially, it is necessary to give the following lemma, which states that any smooth solution to \eqref{1-1}-\eqref{1-2} will keep to be axisymmetric if the initial data is.

\begin{lemma}\label{lem.B1.0}
Assume that the initial data $(\rho_0, u_0, b_0)$ is axisymmetric, then any smooth solution $(\rho, u, b)$ to the system \eqref{1-1}-\eqref{1-2} is still axisymmetric.
\end{lemma}
\begin{proof}
For any $(x_1, x_2, x_3)\in\mathbb{R}^3$, letting
\begin{align*}
r=\sqrt{x_1^2+x_2^2},\quad\theta=\arctan\frac{x_2}{x_1},\quad z=x_3.
\end{align*}
to be the cylindrical coordinate,
\begin{align*}
	e_r(\theta)=(\cos\theta, \sin\theta, 0),\quad e_{\theta}(\theta)=(-\sin\theta, \cos\theta, 0),\quad e_z(\theta)=(0, 0, 1),
\end{align*}
to be the standard basis vectors in the cylindrical coordinate, $\top$ to be the transpose and for every $\theta\in\mathbb{R}$, defining the following rotation matrix,
\begin{align*}
	R_\theta=\begin{pmatrix}
		e_r(\theta)\\
		e_{\theta}(\theta)\\
		e_z(\theta)
	\end{pmatrix}.
\end{align*}

The main thing to prove is the rotation invariance of system \eqref{1-1} under $R_\theta$. To this end, we set
\begin{align*}
&\tilde{\rho}=\rho(y(x))=\rho(xR_\theta^\top),\quad \tilde{u}=u(y(x))R_\theta= u(xR_\theta^\top)R_\theta,\nonumber\\
&\tilde{p}=p(y(x))=p(xR_\theta^\top),\quad \tilde{b}=b(y(x))R_\theta^\top= b(xR_\theta^\top)R_\theta,
\end{align*}
and assume $(\tilde{\rho}, \tilde{u}, \tilde{b})$ to be the smooth solution with axisymmetric initial data $(\rho_0, u_0, b_0)$. Performing some basic calculations, it is clear that
\begin{align}\label{2.1-2}
\nabla\tilde{\rho}=&(\partial_{y_1}\rho\cos\theta-\partial_{y_2}\rho\sin\theta,\,\, \partial_{y_1}\rho\sin\theta+\partial_{y_2}\rho\cos\theta,\,\, \partial_{y_3}\rho)\nonumber\\
=&(\partial_{y_1}\rho,\,\, \partial_{y_2}\rho,\,\, \partial_{y_3}\rho)\begin{pmatrix}
	\cos\theta&\sin\theta&0\\
	-\sin\theta&\cos\theta&0\\
	0&0&1
\end{pmatrix}=\nabla\rho R_\theta,
\end{align}

\begin{align}\label{2.1-20}
	\text{div}\,\tilde{u}=&\nabla\cdot\left( u(y(x)R_\theta)\right)=\nabla\cdot(u^1\cos\theta-u^2\sin\theta,\,\, u^1\sin\theta+u^2\cos\theta,\,\, u^3)\nonumber\\
	=&\left(\partial_{y_1}u^1\cos\theta-\partial_{y_2}u^1\sin\theta\right)\cos\theta-\left(\partial_{y_1}u^2\cos\theta-\partial_{y_2}u^2\sin\theta\right)\sin\theta\\
	&+\left(\partial_{y_1}u^1\sin\theta+\partial_{y_2}u^1\cos\theta\right)\sin\theta+\left(\partial_{y_1}u^2\sin\theta-\partial_{y_2}u^2\cos\theta\right)\cos\theta+\partial_{y_3}u^3\nonumber\\
	=&\text{div}\,{u}=0,\nonumber
\end{align}
and $\nabla\tilde{p}=\nabla pR_\theta$, $\text{div}\,{b}=0$. Similarly, we have
\begin{align}\label{2.1-4}
\nabla\tilde{u}=&\nabla\left( u(y(x)R_\theta)\right)=\begin{pmatrix}
	\partial_{x_1}u^1& \partial_{x_2}u^1& \partial_{x_3}u^1\\
	\partial_{x_1}u^2& \partial_{x_2}u^2& \partial_{x_3}u^2\\
	\partial_{x_1}u^3& \partial_{x_2}u^3& \partial_{x_3}u^3
\end{pmatrix}^\top R_\theta\\
=&\begin{pmatrix}
\partial_{y_1}u^1\cos\theta-\partial_{y_2}u^1\sin\theta& \partial_{y_1}u^1\sin\theta+\partial_{y_2}u^1\cos\theta& \partial_{y_3}u^1\\
\partial_{y_1}u^2\cos\theta-\partial_{y_2}u^2\sin\theta& \partial_{y_1}u^2\sin\theta+\partial_{y_2}u^2\cos\theta& \partial_{y_3}u^2\\
\partial_{y_1}u^3\cos\theta-\partial_{y_2}u^3\sin\theta& \partial_{y_1}u^3\sin\theta+\partial_{y_2}u^3\cos\theta& \partial_{y_3}u^3
\end{pmatrix}^\top R_\theta=R^\top_\theta\nabla uR_\theta,\nonumber
\end{align}
\begin{align}\label{2.1-5}
\triangle\tilde{u}=&\nabla\cdot\nabla\left(u(y(x)R_\theta)\right)=\begin{pmatrix}
	(\partial^2_{y_1}+\partial^2_{y_2}+\partial^2_{y_3})u^1\cos\theta+(\partial^2_{y_1}+\partial^2_{y_2}+\partial^2_{y_3})u^1\sin\theta\\
	-(\partial^2_{y_1}+\partial^2_{y_2}+\partial^2_{y_3})u^1\sin\theta+(\partial^2_{y_1}+\partial^2_{y_2}+\partial^2_{y_3})u^1\cos\theta\\
	(\partial^2_{y_1}+\partial^2_{y_2}+\partial^2_{y_3})u^3
\end{pmatrix}^\top\nonumber\\
=&\triangle u(y(x))\begin{pmatrix}
	\cos\theta&\sin\theta&0\\
	-\sin\theta&\cos\theta&0\\
	0&0&1
\end{pmatrix}=\triangle u(y(x))R_\theta,
\end{align}
and therefore
\begin{align}\label{2.1-6}
\nabla\tilde{b}=R^\top_\theta\nabla bR_\theta,\quad\triangle\tilde{b}=\triangle b(y(x))R_\theta.
\end{align}

Inserting \eqref{2.1-2}-\eqref{2.1-6} into the system \eqref{1-1}, it follows that
\begin{align}\label{2.1-7}
\left\{
\begin{aligned}
&\rho_t+ uR_\theta\cdot\nabla\rho R_\theta=0,\\
&\rho u_t R_\theta+\rho uR_\theta\cdot R^\top_\theta\nabla uR_\theta-\triangle uR_\theta+\nabla pR_\theta=bR_\theta\cdot R^\top_\theta\nabla bR_\theta,\\
&b_tR_\theta-\triangle bR_\theta+uR_\theta\cdot R^\top_\theta\nabla bR_\theta-bR_\theta\cdot R^\top_\theta\nabla uR_\theta=0,\\
&\text{div}\,{u}=\text{div}\,{b}=0.
\end{aligned}
\right.
\end{align}
Due to
\begin{align}\label{2.1-8}
 uR_\theta\cdot \nabla\rho R_\theta=\begin{pmatrix}
u_1\cos\theta-u_2\sin\theta\\
u_1\sin\theta+u_2\cos\theta\\
u_3
\end{pmatrix}^\top\cdot\begin{pmatrix}
\partial_1\rho\cos\theta-\partial_2\rho\sin\theta\\
\partial_1\rho\sin\theta+\partial_2\rho\cos\theta\\
\partial_3\rho
\end{pmatrix}^\top=u\cdot\nabla\rho,
\end{align}
and
\begin{align}\label{2.1-9}
u R_\theta\cdot R_\theta^\top\nabla=u\cdot\nabla,\quad b R_\theta\cdot R^\top_\theta\nabla=b\cdot\nabla,
\end{align}
through multiplying the matrix $R^\top_\theta$ on the right sides of both $\eqref{2.1-7}_2$ and $\eqref{2.1-7}_3$, we can update
the system \eqref{2.1-7} as \eqref{1-1}.

As a consequence, according to the axisymmetry of initial data $(\rho_0, u_0, b_0)$ and the uniqueness of smooth solutions, we can deduce that $\tilde{\rho}=\rho$, $\tilde{u}=u$, $\tilde{b}=b$, which means that $(\rho, u, b)$ is axisymmetric.
\end{proof}

\vskip .1in
Next, for the axisymmetric fluid flowing in the exterior of a cylinder, we are able to establish the following critical Sobolev inequality of logarithmic type involving density, that plays the first  important role in the a priori estimates.
\vskip .1in

\begin{lemma}\label{lem.B3}
Let $\Omega$ be the exterior of a cylinder and $0\leq\rho\leq\bar{\rho}$. Suppose that $u\in H_0^1(\Omega)$ is an axisymmetric vector field, then there exists a genuine constant $C$, such that
\begin{align}\label{2.3-1}
\|\sqrt{\rho}u\|^2_{L^4(\Omega)}
\leq C(\bar{\rho})\left(1+\|\sqrt{\rho}u\|_{L^2(\Omega)}\right)\|u\|_{H^1(\Omega)}\sqrt{\ln\left(2+\|u\|^2_{H^1(\Omega)}\right)}.
\end{align}
\begin{proof} First of all, because $u$ is axisymmetic, we have $u=u^re_r+u^{\theta}e_{\theta}+u^ze_z$. Moreover, considering that $e_r$, $e_{\theta}$ and $e_z$ are orthogonal, to prove \eqref{2.3-1}, it suffices to certify that it holds true for one component, without loss of generality, we choose $\sqrt{\rho}u^\theta e_\theta$ here. According to the definition,
	\begin{align}\label{2.3-20}
		&I=\|\sqrt{\rho}u^\theta e_\theta\|^4_{L^4(\Omega)}=2\pi\int_{-\infty}^{+\infty}\int_1^{+\infty}(r^{\frac{1}{4}}\sqrt{\rho}u^\theta)^4\,drdz,
	\end{align}
from which we can discover that the integrand $(r^{\frac{1}{4}}\sqrt{\rho}u^\theta)^4$ is a function of two variables $r$ and $z$ in the exterior of a unit circle, i.e. $B^{C}$. Therefore, by setting $\tilde{\nabla}=(\p_r,\p_z)$ and applying Lemma \ref{lem.B30}, it follows that
\begin{align}\label{2.3-2}
I\leq&C(\bar{\rho})\left[1+\left(\int_{-\infty}^{+\infty}\int_1^{+\infty}(r^{\frac{1}{4}}\sqrt{\rho}u^\theta)^2\,drdz\right)^{\frac{1}{2}}\right]^2\times\int_{-\infty}^{+\infty}\int_1^{+\infty}\Big((r^{\frac{1}{4}}u^\theta)^2\nonumber\\
&+|\tilde{\nabla}(r^{\frac{1}{4}}u^\theta)|^2\Big)\,drdz\times
\ln\left[2+\int_{-\infty}^{+\infty}\int_1^{+\infty}\left((r^{\frac{1}{4}}u^\theta)^2+|\tilde{\nabla}(r^{\frac{1}{4}}u^\theta)|^2\right)\,drdz\right]\nonumber\\
\leq&C(\bar{\rho})\left[1+\left(\int_{-\infty}^{+\infty}\int_1^{+\infty}(r^{\frac{1}{4}}\sqrt{\rho}u^\theta)^2drdz\right)^{\frac{1}{2}}\right]^2\times\int_{-\infty}^{+\infty}\int_1^{+\infty}\Big((r^{\frac{1}{4}}u^\theta)^2+|\tilde{\nabla}u^\theta|^2r^{\frac{1}{2}}\nonumber\\
&+\Big|\frac{u^\theta}{r}\Big|^2r^{\frac{1}{2}}\Big)\,drdz\times
\ln\left[2+\int_{-\infty}^{+\infty}\int_1^{+\infty}\left((r^{\frac{1}{4}}u^\theta)^2+|\tilde{\nabla}u^\theta|^2r^{\frac{1}{2}}+\Big|\frac{u^\theta}{r}\Big|^2r^{\frac{1}{2}}\right)\,drdz\right]\nonumber\\
\leq&C(\bar{\rho})\left(1+\|\sqrt{\rho}u^\theta\|_{L^2(\Omega)}\right)^2\left(\|u^\theta\|^2_{L^2(\Omega)}+\|\nabla (u^\theta e^\theta)\|^2_{L^2(\Omega)}\right)\\
&\times\ln\left(2+\|u^\theta\|^2_{L^2(\Omega)}+\|\nabla (u^\theta e^\theta)\|^2_{L^2(\Omega)}\right)\nonumber\\
\leq&C(\bar{\rho})\left(1+\|\sqrt{\rho}u^\theta e^\theta\|_{L^2(\Omega)}\right)^2\|u^\theta e^\theta\|^2_{H^1(\Omega)}\ln\left(2+\|u^\theta e^\theta\|^2_{H^1(\Omega)}\right),\nonumber
\end{align}
where we have used the fact $r^{-\eta}\leq1$ for any $\eta>0$ in the third inequality. Similarly, \eqref{2.3-2} also holds for $u^re_r$ and $u^ze_z$. Finally, through summing them up, we can finish all the proof.
\end{proof}
\end{lemma}

Subsequently, unlike ordinary 3D fluids, we can establish the following Gagliardo-Nirenberg inequalities specially for axisymmetric fluids flowing in the exterior of a cylinder and the key indices is as same as 2D flows, that plays the second important role in the a priori estimates.

\begin{lemma} \label{newgn}
Assume that $\Omega$ is the exterior of a cylinder, $ 1 \leq p, q, r \leq \infty $, $ \alpha > 0 $ and $j< m $ be nonnegative integers such that
\begin{align}\label{re1}
\frac{1}{p}=\frac{j}{2}+\alpha\left(\frac{1}{r}-\frac{m}{2}\right)+\left(1-\alpha\right)\frac{1}{q}, \qquad \frac{j}{m} \leq \alpha \leq 1,	
\end{align}
Then for any axisymmetric vector field $u\in L^q(\Omega)$ with $m^{\rm th}$ derivative in
$ L^r(\Omega) $ also has $j^{\rm th}$ derivative in
$ L^p (\Omega) $. Furthermore, it holds that
	\begin{align}\label{new2-2}
		\|\nabla^ju\|_{L^p(\Omega)}\leq \tilde{C}\|\nabla^mu\|_{L^r(\Omega)}^\alpha\|u\|_{L^q(\Omega)}^{1-\alpha},
	\end{align}
	where the constant $\tilde{C}$ depends only on $\Omega$, $m$, $j$, $q$, $r$, $\alpha$.
	\begin{proof} To avoid repetition, we only present the proof of $j=0$ and the cases for $j\geq1$ is similar. Considering that $u$ is axisymmetic, we can rewrite $u$ as  $u=u^re_r+u^{\theta}e_{\theta}+u^ze_z$ firstly. In addition, noticing that $e_r$, $e_{\theta}$ and $e_z$ are orthogonal, to prove \eqref{new2-2} with $j=0$, it suffices to certify it holds for one component. Without loss of generality, we take $u^\theta e_\theta$ for example and it is clear that
		\begin{align}\label{new2.3-20}
			&II=\|u^\theta e_\theta\|^p_{L^p(\Omega)}=2\pi\int_{-\infty}^{+\infty}\int_1^{+\infty}(r^{\frac{1}{p}}u^\theta)^p\,drdz.
		\end{align}
		According to the definition of axisymmetric flow, the integrand  $(r^{\frac{1}{p}}u^\theta)^p$ is a function of two variables $r$ and $z$ only and therefore we can regard it as integrating in the exterior of a unit circle, namely $B^{C}$. Then by applying Lemma \ref{lem.B2} for the 2D exterior domain, it holds that
		\begin{align}\label{newgn1}
			&II\\
			\leq&C\left(\int_{-\infty}^{+\infty}\int_1^{+\infty}|\tilde{\nabla}^m(r^{\frac{1}{p}}u^\theta)|^r\,drdz\right)^{\f{p\alpha}{r}}\left(\int_{-\infty}^{+\infty}\int_1^{+\infty}|(r^{\frac{1}{p}}u^\theta)|^q\,drdz\right)^{\f{p(1-\alpha)}{q}}.\notag
		\end{align}
	Recalling $\tilde{\nabla}=(\p_r,\p_z)$, there are three main parts in $|\tilde{\nabla}^m(r^{\frac{1}{p}}u^\theta)|^r$. The first are $|r^{\frac{1}{p}}\p_r^mu^\theta|^r$	and $|r^{\frac{1}{p}}\p_z^mu^\theta|^r$, which can be bounded by $|\tilde{\nabla}^mu^\theta|^rr^{\frac{r}{p}}$ apparently. The second is $|(\f1p-m)r^{\frac{1}{p}-m-1}u^\theta|^r$. Due to $\nabla={ e}_r\p_r+\f1r{ e}_\theta\p_\theta+{ e}_z\p_z$ and $\p_r{ e}_\theta=0$, $\p_\theta{ e}_\theta=-{ e}_r$,  $\p_z{ e}_\theta=0$, $|(\f1p-m)r^{\frac{1}{p}-m-1}u^\theta|^r$ can be bounded by $C|r^{\frac{1}{p}}\nabla^m(u^\theta{ e}_\theta)|^r$.	The third are the mixing terms of first and second parts, that can be bounded by $C|r^{\frac{1}{p}}\nabla^m(u^\theta{ e}_\theta)|^r$ also. From here, we can update \eqref{newgn1} as	
		\begin{align}\label{newgn2}
		&II\\
		\leq& C\left(\int_{-\infty}^{+\infty}\int_1^{+\infty}|{\nabla}^m(u^\theta{ e}_\theta)|^rr^{\frac{r}{p}-1}r\,drdz\right)^{\f{p\alpha}{r}}\left(\int_{-\infty}^{+\infty}\int_1^{+\infty}|(u^\theta e_\theta)|^qr^{\frac{q}{p}-1}r\,drdz\right)^{\f{p(1-\alpha)}{q}}.\notag
	\end{align}	
Based on the relation \eqref{re0}	for $j=0$ and $n=2$, we can deduce that $r<p$ and $q<p$, which together with $r>1$ further implies $r^{\frac{r}{p}-1}<1$ and $r^{\frac{q}{p}-1}<1$. Thanks to this, we are able to obtain \eqref{new2-2} for $j=0$ and $u=u^\theta e_\theta$. Similarly, \eqref{new2-2} also holds for $j=0$ and $u^re_r$ and $u^ze_z$. In the end, by adding them up, we can finish all the proof.
\end{proof}	
\end{lemma}

Because the case $p=4$ will be used frequently in current paper, we list it  separately for convenience,
see also \cite{JTLiu2024}.

\begin{lemma}\label{lem.B4} Let $\Omega$ be the exterior of a cylinder and suppose that $u\in H^1(\Omega)$ is an axisymmetric vector field, then there exists the genuine constant ${\hat{C}}$ such that
	\begin{equation}\label{2.4-1}
		\|u\|_{L^{4}(\Omega)}\leq {\hat{C}}\|u\|^\frac{1}{2}_{L^{2}(\Omega)}\|\nabla u\|^\frac{1}{2}_{L^{2}(\Omega)},
	\end{equation}
	and
	\begin{equation}\label{new2.4-1}
		\|\nabla u\|_{L^{4}(\Omega)}\leq {\hat{C}}\|\nabla u\|^\frac{1}{2}_{L^{2}(\Omega)}\|\nabla^2 u\|^\frac{1}{2}_{L^{2}(\Omega)}.
	\end{equation}
\end{lemma}

\section{Global well-posedness of strong solutions}
\subsection{A priori estimates}
In this subsection, we will establish the sufficient a priori estimates of local axisymmetric strong solutions $(\rho, u, b)$ to the system \eqref{1-1}-\eqref{1-2}. Specifically, for given initial data $(\rho_0, u_0, b_0)$ satisfying \eqref{1-7}, we intend to work on them step by step, that will be given
in Lemmas \ref{lem.B3.01}-\ref{lem.B3.4}.\\

First, making use of the transport equation $\eqref{1-1}_1$ and $\mathrm{div}\,u=0$, we can directly obtain the following basic estimates.
\begin{lemma}\label{lem.B3.01}
For $(x, t)\in\Omega\times[0, T]$, it holds that
\begin{equation}\label{3.1-1}
\sup\limits_{t\in[0, T]}\|\rho\|_{L^\infty}=\|\rho_0\|_{L^\infty}\triangleq\bar{\rho},\quad
\sup\limits_{t\in[0, T]}\|\rho-\bar{\rho}\|_{L^\frac{3}{2}}\leq\|\rho_0-\bar{\rho}\|_{L^\frac{3}{2}}.
\end{equation}
\end{lemma}
\begin{lemma}\label{lem.B3.1}
There exists a genuine constant C depending only on  $\bar{\rho}$,  $\|\rho_0-\bar{\rho}\|_{L^{\frac{3}{2}}}$, $\|u_0\|_{L^2}$ and $\|b_0\|_{L^2}$ such that
\begin{equation}\label{3.1-2}
\sup\limits_{t\in[0, T]}\left(\|\sqrt{\rho}u\|_{L^2}^2+\|b\|_{L^2}^2\right)+\int_0^T\left(\|\nabla u\|_{L^2}^2+\|\nabla b\|_{L^2}^2\right)\,dt\leq C.
\end{equation}
\end{lemma}
\begin{proof}
Taking inner product of $\eqref{1-1}_2$ and $\eqref{1-1}_3$ with $u$ and $b$ respectively and then integrating by parts, it follows that
\begin{align}\label{3.1-5}
\frac{1}{2}\frac{d}{dt}\left(\|\sqrt{\rho}u\|_{L^2}^2+\|b\|^2_{L^2}\right)+\|\nabla u\|^2_{L^2}+\|\nabla b\|^2_{L^2}=0,
\end{align}
which implies, after integrating in time, that
\begin{align}\label{3.1-6}
\sup\limits_{t\in[0, T]}\|\sqrt{\rho}u,b\|_{L^2}^2+\int_0^T\|\nabla u,\nabla b\|^2_{L^2}\,dt
\leq \left(\|\sqrt{\rho_0}u_0\|^2_{L^2}+\|b_0\|^2_{L^2}\right),
\end{align}
and \eqref{3.1-2}. Noticing \eqref{3.1-1}, \eqref{3.1-6}, $u|_{\partial\Omega}=0$ and by extending to zero for $x\in{\Omega}^{c}$, one can obtain the Sobolev embedding inequality
\begin{equation}\label{poin}
\|u\|_{L^6}\leq\|\nabla u\|_{L^2},
\end{equation}
and then we have
\begin{align}\label{3.1-7}
&\bar{\rho}\int|u|^2\,dx=\int\rho|u|^2\,dx-\int\left(\rho-\bar{\rho}\right)|u|^2\,dx\nonumber\\
&\leq \|\sqrt{\rho_0}u_0\|^2_{L^2}+\|\rho-\bar{\rho}\|_{L^{\frac{3}{2}}}\|u\|^2_{L^6}
\leq C\left(1+\|\nabla u\|^2_{L^2}\right).
\end{align}
Thus, the proof is finished.
\end{proof}

\vskip .1in
Subsequently, we will make full use of the axisymmetric property of solutions in the exterior of a cylinder to establish the first key estimates, i.e. the $L^{\infty}\left([0, +\infty); H^1\right)$ norms of velocity and magnetic fields.
\vskip .1in

\begin{lemma}\label{lem.B3.2}
Suppose that $(\rho, u, b)$ is an axisymmetric solution to the system \eqref{1-1}-\eqref{1-2} and \eqref{1-2boundary}, then there exists a genuine constant $C$ depending only on $\bar{\rho}$, $\|\rho_0-\bar{\rho}\|_{L^{\f32}}$,  $\| u_0\|_{H^1}$ and $\|b_0\|_{H^1}$ such that
\begin{align}\label{3.2-1}
	\sup\limits_{t\in[0, T]}\left(\| u,b\|^2_{H^1}+\| b\|^4_{L^4}\right)
	+\int_0^T\left(\|\sqrt{\rho}u_t\|^2_{L^2}+\|b_t\|^2_{L^2}+\|{\nabla}^2 b\|^2_{L^2}+\|\nabla u\|^2_{H^1}\right)dt\leq C,
\end{align}
and
\begin{align}\label{3.2-01}
	\sup\limits_{t\in[0, T]}t\left(\|\nabla u\|^2_{L^2}+\|\nabla b\|^2_{L^2}+\| b\|^4_{L^4}\right)
	+\int_0^Tt\left(\|\sqrt{\rho}u_t\|^2_{L^2}+\|b_t\|^2_{L^2}+\|{\nabla}^2 b\|^2_{L^2}\right)\,dt\leq C.
\end{align}
\end{lemma}
\begin{proof}
Taking inner product of $\eqref{1-1}_2$ with $u_t$, using $u|_{\partial\Omega}=0$ and integrating by parts, there holds
\begin{align}\label{3.2-3}
\frac{1}{2}\frac{d}{dt}\int|\nabla u|^2\,dx+\int\rho|u_t|^2\,dx
=-\int\rho u\cdot\nabla u\cdot u_t\,dx+\int b\cdot\nabla b\cdot u_t\,dx.
\end{align}
By H\"{o}lder, Lemma \ref{lem.B4} and Young inequalities, it yields that
\begin{align}\label{3.2-4}
|-\int\rho u\cdot\nabla u\cdot u_t\,dx|\leq\frac{1}{2}\|\sqrt{\rho}u_t\|^2_{L^2}+C\|\sqrt{\rho}u\|^2_{L^4}\|\nabla u\|_{L^2}\|\nabla^2 u\|_{L^2}.
\end{align}
Utilizing integration by parts together with $\mathrm{div}\,b=0$, $b|_{\partial\Omega}=0$, $u|_{\partial\Omega}=0$, Lemma \ref{lem.B4}, H\"{o}lder, Lemma \ref{newgn} and Young inequalities and \eqref{poin}, one has
\begin{align}\label{3.2-5}
&\int b\cdot\nabla b\cdot u_t\,dx\nonumber\\
=&-\frac{d}{dt}\int b\cdot\nabla u\cdot b\,dx+\int\left(\triangle b-u\cdot\nabla b+b\cdot\nabla u\right)\cdot\nabla u\cdot b\,dx\nonumber\\
&+\int b\cdot\nabla u\cdot \left(\triangle b-u\cdot\nabla b+b\cdot\nabla u\right)\,dx\nonumber\\
\leq&-\frac{d}{dt}\int b\cdot\nabla u\cdot b\,dx+C\left(\|{\nabla}^2 b\|_{L^2}\|\nabla u\|_{L^4}\|b\|_{L^4}+\|\nabla u\|^2_{L^4}\|b\|^2_{L^4}\right)\\
&+C\|u\|_{L^\infty}\|\nabla b\|_{L^2}\|\nabla u\|_{L^4}\|b\|_{L^4}\nonumber\\
\leq&-\frac{d}{dt}\int b\cdot\nabla u\cdot b\,dx+\frac{1}{4}\|{\nabla}^2 b\|^2_{L^2}+C\|\nabla u\|_{L^2}\|\nabla^2 u\|_{L^2}\|b\|^2_{L^4}\nonumber\\
&+C\|u\|^\f13_{L^6}\|\nabla u\|^\f23_{L^6}\|\nabla b\|_{L^2}\|\nabla u\|^{\frac{1}{2}}_{L^2}\|\nabla^2 u\|^{\frac{1}{2}}_{L^2}\|b\|_{L^4}\nonumber\\
\leq&-\frac{d}{dt}\int b\cdot\nabla u\cdot b\,dx+\frac{1}{4}\|{\nabla}^2 b\|^2_{L^2}+C\|\nabla u\|_{L^2}\|\nabla^2 u\|_{L^2}\|b\|^2_{L^4}\nonumber\\
&+C\|\nabla b\|_{L^2}\|b\|_{L^4}\|\nabla u\|_{L^2}\|\nabla u\|_{H^1}.\nonumber
\end{align}
Substituting \eqref{3.2-4} and \eqref{3.2-5} into \eqref{3.2-3}, we derive
\begin{align}\label{3.2-6}
&\frac{1}{2}\frac{d}{dt}\|\nabla u\|^2_{L^2}+\frac{1}{2}\|\sqrt{\rho}u_t\|^2_{L^2}\nonumber\\
\leq&-\frac{d}{dt}\int b\cdot\nabla u\cdot b\,dx+\frac{1}{4}\|{\nabla}^2b\|^2_{L^2}+C\|\sqrt{\rho}u\|^2_{L^4}\|\nabla u\|_{L^2}\|\nabla^2 u\|_{L^2}\\
&+C\|\nabla u\|_{L^2}\|\nabla^2 u\|_{L^2}\|b\|^2_{L^4}+C\|\nabla b\|_{L^2}\|b\|_{L^4}\|\nabla u\|_{L^2}\|\nabla u\|_{H^1}.\nonumber
\end{align}

Taking inner product of $\eqref{1-1}_3$ with $\triangle b$, using integrating by parts, Lemma \ref{lem.B4}, H\"{o}lder and Younng inequalities, it follows that
\begin{align}\label{3.2-7}
&\frac{1}{2}\frac{d}{dt}\|\nabla b\|^2_{L^2}+\|\nabla^2b\|^2_{L^2}=\int u\cdot\nabla b\cdot\triangle b\,dx-\int b\cdot\nabla u\cdot\triangle b\,dx\nonumber\\
\leq&\int|\nabla u||\nabla b|^2\,dx+\int|b||\nabla u||\triangle b|\,dx\leq C\left(\|\nabla u\|_{L^2}\|\nabla b\|^2_{L^4}+\|{\nabla}^2b\|_{L^2}\|\nabla u\|_{L^4}\|b\|_{L^4}\right)\nonumber\\
\leq&C\left(\|\nabla u\|_{L^2}\|\nabla b\|_{L^2}\|\nabla^2 b\|_{L^2}+\|{\nabla}^2b\|_{L^2}\|\nabla u\|^{\frac{1}{2}}_{L^2}\|\nabla^2 u\|^{\frac{1}{2}}_{L^2}\|b\|_{L^4}\right)\\
\leq&\frac{1}{4}\|\nabla^2b\|^2_{L^2}+C\left[\|\nabla u\|^2_{L^2}\|\nabla b\|^2_{L^2}+\|\nabla u\|_{L^2}\|\nabla^2 u\|_{L^2}\|b\|^2_{L^4}\right].\nonumber
\end{align}
Using Lemma \ref{lem.B4} and \eqref{3.1-2}, it follows that
\begin{align*}
	\int|b\cdot\nabla u\cdot b|\,dx\leq\frac{1}{4}\|\nabla u\|^2_{L^2}+C_1\|b\|^4_{L^4}\leq\frac{1}{4}\|\nabla u\|^2_{L^2}+C\|\nabla b\|^2_{L^2}.
\end{align*}
Next, we introduce a new quantity
$A(t)\triangleq\|\nabla u\|^2_{L^2}+\|\nabla b\|^2_{L^2}+\int b\cdot\nabla u\cdot b\,dx$ satisfying
\begin{align}\label{3.2-9}
	\frac{3}{4}\left(\|\nabla u\|^2_{L^2}+\|\nabla b\|^2_{L^2}\right)-C_1\|b\|^4_{L^4}\leq A(t)\leq C\|\nabla u\|^2_{L^2}+C\|\nabla b\|^2_{L^2},
\end{align}
and then update \eqref{3.2-6} and \eqref{3.2-7} as
\begin{align}\label{3.2-8}
&A'(t)+\|\sqrt{\rho}u_t\|^2_{L^2}+\|\nabla^2b\|^2_{L^2}\nonumber\\
\leq&C\left[\|\sqrt{\rho}u\|^2_{L^4}\|\nabla u\|_{L^2}\|\nabla^2 u\|_{L^2}+\|\nabla u\|^2_{L^2}\|\nabla b\|^2_{L^2}\right]\\
&+C\|\nabla u\|_{L^2}\|\nabla^2 u\|_{L^2}\|b\|^2_{L^4}+C\|\nabla b\|_{L^2}\|b\|_{L^4}\|\nabla u\|_{L^2}\|\nabla u\|_{H^1}.\nonumber
\end{align}

Taking inner product of $\eqref{1-1}_3$ with $|b|^2b$, and integrating by parts, using H\"{o}lder inequality and Lemma \ref{lem.B4}, we have
\begin{align}\label{3.2-10}
&\frac{1}{4}\frac{d}{dt}\|b\|^4_{L^4}+\||\nabla b||b|\|^2_{L^2}+\frac{1}{2}\|\nabla|b|^2\|^2_{L^2}\nonumber\\
\leq&C\|\nabla u\|_{L^2}\||b|^2\|^2_{L^4}\leq C\|\nabla u\|_{L^2}\||b|^2\|_{L^2}\|\nabla|b|^2\|_{L^2}\\
\leq&\frac{1}{4}\|\nabla|b|^2\|^2_{L^2}+C\|\nabla u\|^2_{L^2}\|b\|^4_{L^4},\nonumber
\end{align}
which implies, after utilizing Gr\"{o}nwall inequality and  \eqref{3.1-2}, that
\begin{align}\label{3.2-11}
&\sup\limits_{t\in[0, T]}\|b\|^4_{L^4}+\int_0^T\||\nabla b||b|\|^2_{L^2}\,dt+\int_0^T\|\nabla|b|^2\|^2_{L^2}\,dt
\nonumber\\
\leq&C\|b_0\|^4_{L^4}e^{\int_0^T\|\nabla u\|^2_{L^2}\,dt}\leq C.
\end{align}
Multiplying on both sides of \eqref{3.2-10} by $t$, it yields that
\begin{align}\label{3.2-010}
	\frac{1}{4}\frac{d}{dt}\left(t\|b\|^4_{L^4}\right)
	\leq Ct\|\nabla u\|^2_{L^2}\|b\|^4_{L^4}+\f14\|b\|^4_{L^4}\leq Ct\|\nabla u\|^2_{L^2}\|b\|^4_{L^4}+C\|b\|^2_{L^2}\|\nabla b\|^2_{L^2},
\end{align}
and then
\begin{align}\label{3.2-0010}
	\sup\limits_{t\in[0, T]}t\|b\|^4_{L^4}\leq C.
\end{align}

From Lemma \ref{lem.B5}, \eqref{3.1-2}, Lemma \ref{newgn} and \eqref{3.2-11}, it deduces that
\begin{align}\label{3.2-13}
\|\nabla^2u\|_{L^2}\leq&C\left(\|\rho u_t+\rho u\cdot\nabla u+b\cdot\nabla b\|_{L^2}\right)\nonumber\\
\leq&C\left(\|\sqrt{\rho}u_t\|_{L^2}+\|\sqrt{\rho}u\|_{L^4}\|\nabla u\|_{L^4}+\||\nabla b||b|\|_{L^2}\right)\\
\leq&C\left(\|\sqrt{\rho}u_t\|_{L^2}+\|\sqrt{\rho}u\|_{L^4}\|\nabla u\|^{\frac{1}{2}}_{L^2}\|\nabla^2 u\|^{\frac{1}{2}}_{L^2}+\||\nabla b||b|\|_{L^2}\right),\nonumber
\end{align}
this is
\begin{align}\label{3.2-013}
\|\nabla^2u\|_{L^2}&\leq C\left[\|\sqrt{\rho}u_t\|_{L^2}+\|\sqrt{\rho}u\|^2_{L^4}\|\nabla u\|_{L^2}+\||\nabla b||b|\|_{L^2}\right]\notag\\
&\leq C\left[\|\sqrt{\rho}u_t\|_{L^2}+\|\sqrt{\rho}u\|^2_{L^4}\|\nabla u\|_{L^2}+\|b\|_{L^4}\|\nabla b\|_{L^2}^{\f12}\|\nabla^2 b\|_{L^2}^{\f12}\right]\\
&\leq C\left[\|\sqrt{\rho}u_t\|_{L^2}+\|\sqrt{\rho}u\|^2_{L^4}\|\nabla u\|_{L^2}+\|\nabla b\|_{L^2}^{\f12}\|\nabla^2 b\|_{L^2}^{\f12}\right]\notag
\end{align}
Substituting \eqref{3.2-013} into \eqref{3.2-8} and making use of \eqref{3.2-11} and Young inequality, it yields that
\begin{align}\label{3.2-14}
&A'(t)+\|\sqrt{\rho}u_t\|^2_{L^2}+\|\nabla^2b\|^2_{L^2}\\
\leq&\frac{1}{2}\||\nabla b||b|\|^2_{L^2}+C\left[\left(1+\|\sqrt{\rho}u\|^4_{L^4}\right)\|\nabla u\|^2_{L^2}+\left(1+\|\nabla u\|^2_{L^2}\right)\|\nabla b\|^2_{L^2}\right],\nonumber
\end{align}
which further implies, after multiplying $4C_1$ on both sides of \eqref{3.2-10}, adding the resultant with \eqref{3.2-14} and using Lemma \ref{lem.B3} and \eqref{3.1-7}, that
\begin{align}\label{3.2-16}
&\frac{d}{dt}\left(A(t)+C_1\|b\|^4_{L^4}\right)+\|\sqrt{\rho}u_t\|^2_{L^2}+\|{\nabla}^2 b\|^2_{L^2}+\||\nabla b||b|\|^2_{L^2}\nonumber\\
\leq&C\left[\left(1+\|\sqrt{\rho}u\|^4_{L^4}\right)\|\nabla u\|^2_{L^2}+\left(1+\|\nabla b\|^2_{L^2}\right)\left(1+\|\nabla u\|^2_{L^2}\right)\right]\\
\leq&C\|\nabla u\|^2_{L^2}\left(2+\|\nabla u\|^2_{L^2}\right)\ln\left(2+\|\nabla u\|^2_{L^2}\right)+C\left(1+\|\nabla u\|^2_{L^2}\right)\|\nabla b\|^2_{L^2}.\nonumber
\end{align}
To estimate \eqref{3.2-16}, we need to employ Gr\"{o}nwall
inequality and as a preparation, it is necessary to set
\begin{align*}
f(t)\triangleq A(t)+C_1\|b\|^4_{L^4}+2,\quad g(t)\triangleq\|\nabla u\|^2_{L^2}+\|\nabla b\|^2_{L^2}+2.
\end{align*}
Then, according to \eqref{3.2-16}, we have
\begin{align}\label{3.2-17}
f'(t)\leq Cg(t)f(t)+Cg(t)f(t)\ln(f(t)),
\end{align}
i.e.
\begin{align}\label{3.2-18}
\left(\ln f(t)\right)'\leq Cg(t)+Cg(t)\ln(f(t)),
\end{align}
which yields, after applying Gr\"{o}nwall inequality and \eqref{3.1-2}, that
\begin{align}\label{3.2-019}
\sup\limits_{t\in[0, T]}\ln(f(t))\leq C.
\end{align}

Thus, thanks to \eqref{3.2-9} and \eqref{3.2-11}, we have proved
\begin{align}\label{3.2-19}
\sup\limits_{t\in[0, T]}\left(\|\nabla u\|^2_{L^2}+\|\nabla b\|^2_{L^2}\right)\leq C,
\end{align}
which further implies, after integrating \eqref{3.2-16} in time and employing \eqref{3.1-2} and \eqref{3.1-7}
\begin{align}\label{3.2-20}
\sup\limits_{t\in[0, T]}\left(\| u\|^2_{H^1}+\| b\|^2_{H^1}\right)+\int_0^T\left(\|\sqrt{\rho}u_t\|^2_{L^2}+\|{\nabla}^2 b\|^2_{L^2}+\||b||\nabla b|\|^2_{L^2}\right)\,dt
\leq C.
\end{align}
Utilizing Lemma \ref{lem.B3}, \ref{lem.B3.1} and \eqref{3.2-20}, it follows that
\begin{align}\label{3.2-21}
	\sup\limits_{t\in[0, T]}\|\sqrt{\rho}u\|_{L^4}\leq C,
\end{align}
which further implies, after using \eqref{3.2-21}, \eqref{3.2-013},
\eqref{3.2-20}, that
\begin{align}\label{3.2-22}
\int_0^T\|\nabla^2u\|^2_{L^2}\,dt\leq C.
\end{align}
 Moreover, on the basis of $\eqref{1-1}_3$, Lemma \ref{newgn} and \eqref{3.2-19}, we have
\begin{align}\label{3.3-8.0}
	\|b_t\|^2_{L^2}\leq&C\left(\|\triangle b\|^2_{L^2}+\|u\|^2_{L^6}\|\nabla b\|^2_{L^3}+\|b\|^2_{L^\infty}\|\nabla u\|^2_{L^2}\right)\nonumber\\
	\leq&C\left(\|\nabla^2 b\|^2_{L^2}+\|\nabla u\|^2_{L^2}\|\nabla b\|_{L^2}^{\f43}\|\nabla^2 b\|_{L^2}^{\f23}+\| b\|_{L^6}^{\f43}\|\nabla^2 b\|_{L^2}^{\f23}\|\nabla u\|^2_{L^2}\right)\nonumber\\
	\leq&C\|\nabla^2b\|^2_{L^2}+C\|\nabla u\|^2_{L^2}\|\nabla b\|^2_{L^2},
\end{align}
Thus, combining \eqref{3.2-11}, \eqref{3.2-20}, \eqref{3.2-22} and \eqref{3.3-8.0} leads to \eqref{3.2-1}.

It remains to prove \eqref{3.2-01}, to this end, we multiply \eqref{3.2-16} by $t$ and apply
 \eqref{3.2-19} and \eqref{3.2-9} to get
\begin{align}\label{3.2-23}
&\frac{d}{dt}\left[t\left(A(t)+C_1\|b\|^4_{L^4}\right)\right]+t\left(\|\sqrt{\rho}u_t\|^2_{L^2}+\|{\nabla}^2 b\|^2_{L^2}+\||b||\nabla b|\|^2_{L^2}\right)\nonumber\\
\leq& Ct\left(\|\nabla u\|^2_{L^2}+\|\nabla b\|^2_{L^2}\right)^2+A(t)+C_1\|b\|^4_{L^4}+\|\nabla b\|^2_{L^2}\\
\leq& C\left(\|\nabla u\|^2_{L^2}+\|\nabla b\|^2_{L^2}\right)\left[t\left(A(t)+C_1\|b\|^4_{L^4}\right)\right]+ C\left(\|\nabla u\|^2_{L^2}+\|\nabla b\|^2_{L^2}\right),\nonumber
\end{align}
which leads to
\begin{align}\label{3.2-01.0}
	\sup\limits_{t\in[0, T]}t\left(\|\nabla u\|^2_{L^2}+\|\nabla b\|^2_{L^2}\right)
	+\int_0^Tt\left(\|\sqrt{\rho}u_t\|^2_{L^2}+\|{\nabla}^2 b\|^2_{L^2}+\||b||\nabla b|\|^2_{L^2}\right)\,dt\leq C,
\end{align}
after employing Gr\"{o}nwall inequality, \eqref{3.1-2} and \eqref{3.2-0010}. Finally, by multiplying \eqref{3.3-8.0} with $t$ and applying \eqref{3.2-01.0} and \eqref{3.2-19}, we can finish the proof of this lemma.
\end{proof}

\vskip .1in
Next, with the help of Lemma \ref{lem.B3.2}, we are able to establish the following various time and spatial derivatives and time-weighted estimates of velocity and magnetic fields.
\vskip .1in

\begin{lemma}\label{lem.B3.3}
Suppose that $(\rho, u, b)$ is an axisymmetric solution to the system \eqref{1-1}-\eqref{1-2} and \eqref{1-2boundary}, then there exists a genuine constant $C$ depending only on $\bar{\rho}$, $\|\rho_0-\bar{\rho}\|_{L^{\f32}}$,  $\| u_0\|_{H^1}$ and $\|b_0\|_{H^1}$ such that
\begin{align}\label{3.3-1}
\sup\limits_{t\in[0, T]}t^i\left(\|\sqrt{\rho}u_t\|^2_{L^2}+\|b_t\|_{L^2}^2\right)+\int_0^Tt^i\left(\|\nabla u_t\|^2_{L^2}+\|\nabla b_t\|^2_{L^2}\right)\,dt\leq C,
\end{align}
and
\begin{align}\label{3.3-01}
t^2\left(\|\nabla^2 u\|^2_{L^2}+\|\nabla^2 b\|^2_{L^2}\right)+\int_0^Tt\left(\| u_t\|^2_{L^2}+\|\nabla^2 u\|^2_{L^2}\right)\,dt\leq C,
\end{align}
for $i\in\{1, 2\}$.
\end{lemma}
\begin{proof}
Differentiating $\eqref{1-1}_2$ with $t$, taking inner product of the resultant with $u_t$, making use of integration by parts and $\eqref{1-1}_1$, we get
\begin{align}\label{3.3-3}
&\frac{1}{2}\frac{d}{dt}\|\sqrt{\rho}u_t\|^2_{L^2}+\|\nabla u_t\|^2_{L^2}
=-2\int\rho u\cdot\nabla u_t\cdot u_t\,dx-\int\rho u_t\cdot\nabla u\cdot u_t\,dx\nonumber\\
&-\int\rho u\cdot\nabla\left(u\cdot\nabla u\cdot u_t\right)\,dx+\int b_t\cdot\nabla b\cdot u_t\,dx+\int b\cdot\nabla b_t\cdot u_t\,dx\nonumber\\
\triangleq&\sum\limits_{i=1}^{5}I_i.
\end{align}
According to H\"{o}lder, Gagliardo-Nirenberg and Young inequalities, $I_1$ and $I_2$ can be estimated as
\begin{align}\label{3.3-4}
|&I_1|+|I_2|
\leq2\int|\rho u\cdot\nabla u_t\cdot u_t|\,dx+\int|\rho u_t\cdot\nabla u\cdot u_t|\,dx\nonumber\\
\leq&2\sqrt{\bar{\rho}}\|u\|_{L^\infty}\|\sqrt{\rho}u_t\|_{L^2}\|\nabla u_t\|_{L^2}+\sqrt{\bar{\rho}}\|\sqrt{\rho}u_t\|_{L^2}\|\nabla u\|_{L^3}\|u_t\|_{L^6}\nonumber\\
\leq&C\left(\|u\|^\f23_{L^6}\|\nabla u\|^\f13_{L^6}+\|\nabla u\|^\f23_{L^2}\|\nabla^2 u\|^{\f13}_{L^2}\right)\|\sqrt{\rho}u_t\|_{L^2}\|\nabla u_t\|_{L^2}\\
\leq&\frac{1}{6}\|\nabla u_t\|^2_{L^2}+C\|\nabla u\|^2_{H^1}\|\sqrt{\rho}u_t\|^2_{L^2}.\nonumber
\end{align}
Similarly, the third term can be estimated as
\begin{align}\label{3.3-5}
I_3\leq&\int \rho|u|\left(|u_t||\nabla u|^2+|u||u_t||\nabla^2u|+|u||\nabla u_t||\nabla u|\right)\,dx\\
\leq&\bar{\rho}\left(\|u_t\|_{L^6}\|u\|_{L^6}\|\nabla u\|_{L^6}\|\nabla u\|_{L^2}+\|u\|^2_{L^6}\|u_t\|_{L^6}\|\nabla^2u\|_{L^2}+\|u\|^2_{L^6}\|\nabla u\|_{L^6}\|\nabla u_t\|_{L^2}\right)\nonumber\\
\leq&C\|\nabla u_t\|_{L^2}\|\nabla u\|^2_{L^2}\|\nabla u\|_{H^1}\nonumber\\
\leq&\frac{1}{6}\|\nabla u_t\|^2_{L^2}+C\|\nabla u\|^4_{L^2}\|\nabla u\|^2_{H^1}.\nonumber
\end{align}
For $I_4$ and $I_5$, we can obtain from integration by parts, Lemma \ref{lem.B4} and \eqref{3.2-11} that
\begin{align}\label{3.3-6}
&|I_4|+|I_5|
\leq\int |b_t\cdot\nabla u_t\cdot b|\,dx+\int |b\cdot\nabla u_t\cdot b_t|\,dx\nonumber\\
\leq&C\|b_t\|_{L^4}\|\nabla u_t\|_{L^2}\|b\|_{L^4}\nonumber\\
\leq&C\|\nabla u_t\|_{L^2}\|b_t\|^{\frac{1}{2}}_{L^2}\| \nabla b_t\|^{\frac{1}{2}}_{L^2}\nonumber\\
\leq&\frac{1}{6}\|\nabla u_t\|^2_{L^2}+C(\delta)\|b_t\|^2_{L^2}+\frac{\delta}{2}\|\nabla b_t\|^2_{L^2}.
\end{align}

Substituting \eqref{3.3-4}-\eqref{3.3-6} into \eqref{3.3-3} leads to
\begin{align}\label{3.3-7}
&\frac{d}{dt}\|\sqrt{\rho}u_t\|^2_{L^2}+\|\nabla u_t\|^2_{L^2}\nonumber\\
\leq&C\|\nabla u\|^4_{L^2}\left(\|\sqrt{\rho}u_t\|^2_{L^2}+\|\nabla u\|^2_{H^1}\right)+C(\delta)\|b_t\|^2_{L^2}+\delta\|\nabla b_t\|^2_{L^2}.
\end{align}
Thanks to \eqref{3.2-21}, it holds that
\begin{align}\label{3.3-007}
\|\nabla^2u\|_{L^2}\leq C\left(\|\sqrt{\rho}u_t\|_{L^2}+\|\nabla u\|_{L^2}+\||b||\nabla b|\|_{L^2}\right),
\end{align}
which together with \eqref{3.3-8.0} and \eqref{3.3-7} implies
\begin{align}\label{3.3-9}
\frac{d}{dt}&\|\sqrt{\rho}u_t\|^2_{L^2}+\|\nabla u_t\|^2_{L^2}\nonumber\\
\leq&C\|\nabla u\|^4_{L^2}\left(\|\sqrt{\rho}u_t\|_{L^2}^2+\|\nabla u\|_{L^2}^2+\||b||\nabla b|\|_{L^2}^2\right)+\delta\|\nabla b_t\|^2_{L^2}\nonumber\\
&+C(\delta)\|\nabla^2b\|^2_{L^2}+C(\delta)\|\nabla u\|^2_{L^2}+\|\nabla b\|^2_{L^2}.
\end{align}

Differentiating $\eqref{1-1}_3$ with $t$, taking inner product of the resultant with $b_t$, using integration by parts, Gagliardo-Nirenberg inequality, \eqref{3.1-2} and \eqref{3.2-19}, it yields that
\begin{align}\label{3.3-10}
&\frac{1}{2}\frac{d}{dt}\|b_t\|_{L^2}^2+\|\nabla b_t\|^2_{L^2}
\leq\int\left(|u_t||b|+|u||b_t|\right)|\nabla b_t|\,dx\nonumber\\
\leq&C\left(\|u_t\|_{L^6}\|b\|_{L^3}+\|u\|_{L^6}\|b_t\|_{L^3}\right)\|\nabla b_t\|_{L^2}\nonumber\\
\leq&C\left(\|\nabla u_t\|_{L^2}\|b\|^\f23_{L^2}\|\nabla b\|^\f13_{L^2}+\|\nabla u\|_{L^2}\|b_t\|^\f23_{L^2}\|\nabla b_t\|^\f13_{L^2}\right)\|\nabla b_t\|_{L^2}\nonumber\\
\leq&\frac{1}{2}\|\nabla b_t\|^2_{L^2}+C\|\nabla u_t\|^2_{L^2}+C\|b_t\|^2_{L^2},
\end{align}
which can be updated as, after using \eqref{3.3-8.0},
\begin{align}\label{3.3-11}
&\frac{d}{dt}\|b_t\|_{L^2}^2+\|\nabla b_t\|^2_{L^2}\nonumber\\
\leq&C_2\|\nabla u_t\|^2_{L^2}+C\|\nabla^2b\|^2_{L^2}+C\|\nabla u\|^2_{L^2}\|\nabla b\|^2_{L^2}.
\end{align}
where $C_2$ is a genuine constant.

Choosing $\delta=\frac{1}{4C_2}$, multiplying \eqref{3.3-9}
with $2C_2$ and adding the resultant with \eqref{3.3-11}, we have
\begin{align}\label{3.3-12}
&\frac{d}{dt}\left(2C_2\|\sqrt{\rho}u_t\|^2_{L^2}+\|b_t\|_{L^2}^2\right)+C_2\|\nabla u_t\|^2_{L^2}+\frac{1}{2}\|\nabla b_t\|^2_{L^2}\notag\\
\leq&C\|\nabla u\|^4_{L^2}\left(\|\sqrt{\rho}u_t\|_{L^2}^2+\|\nabla u\|_{L^2}^2+\||b||\nabla b|\|_{L^2}^2\right)\\
&+C\|\nabla^2b\|^2_{L^2}+C\|\nabla u\|^2_{L^2}\|\nabla b\|^2_{L^2},\nonumber
\end{align}
which also implies, after multiplying by $t^i$, that
\begin{align}\label{3.3-13}
&\frac{d}{dt}\left[t^i\left(2C_2\|\sqrt{\rho}u_t\|^2_{L^2}+\|b_t\|_{L^2}^2\right)\right]+t^i\left(C_2\|\nabla u_t\|^2_{L^2}+\f12\|\nabla b_t\|^2_{L^2}\right)\nonumber\\
\leq&Ct^i\|\nabla u\|^4_{L^2}\left(\|\sqrt{\rho}u_t\|_{L^2}^2+\|\nabla u\|_{L^2}^2+\||b||\nabla b|\|_{L^2}^2\right)\nonumber\\
&+Ct^i\left(\|\nabla^2b\|^2_{L^2}+\|\nabla u\|^2_{L^2}\|\nabla b\|^2_{L^2}\right)+Cit^{i-1}\left(\|\sqrt{\rho}u_t\|^2_{L^2}+\|b_t\|_{L^2}^2\right).
\end{align}
Integrating \eqref{3.3-13} over $(0,T)$ and utilizing \eqref{3.1-2}, \eqref{3.2-1} and \eqref{3.2-01}, one can obtain \eqref{3.3-1}.

To prove \eqref{3.3-01}, we can invoke the elliptic theory, Lemma \ref{newgn} and \eqref{3.2-19} to get
\begin{align}\label{3.3-15}
\|\nabla^2 b\|^2_{L^2}\leq&C\left(\|b_t\|^2_{L^2}+\|u\cdot\nabla b\|^2_{L^2}+\|b\cdot\nabla u\|^2_{L^2}\right)\nonumber\\
\leq&C\left(\|b_t\|^2_{L^2}+\|u\|^2_{L^6}\|\nabla b\|^2_{L^3}+\|b\|^2_{L^\infty}\|\nabla u\|^2_{L^2}\right)\nonumber\\
\leq&C\left(\|b_t\|^2_{L^2}+\|\nabla u\|^2_{L^2}\|\nabla b\|_{L^2}^{\f43}\|\nabla^2 b\|_{L^2}^{\f23}+\|b\|_{L^6}^{\f32}\|\nabla^2 b\|_{L^2}^{\f12}\|\nabla u\|^2_{L^2}\right)\\
\leq&C\left(\|b_t\|^2_{L^2}+\|\nabla u\|^4_{L^2}+\|\nabla b\|^4_{L^2}\right)+\frac{1}{2}\|\nabla^2 b\|^2_{L^2},\nonumber
\end{align}
which together with \eqref{3.2-01} and \eqref{3.3-1}  implies
\begin{align}\label{3.3-16}
t^2\|\nabla^2 b\|^2_{L^2}+\int_0^Tt\|\nabla^2 b\|^2_{L^2}\,dt\leq C.
\end{align}
With the help of \eqref{3.2-19}, one can update \eqref{3.2-13} as
\begin{align}\label{3.3-17}
	&\|\nabla^2u\|_{L^2}^2\leq C\left(\|\sqrt{\rho}u_t\|_{L^2}^2+\|u\|_{L^\infty}^2\|\nabla u\|_{L^2}^2+\|b\|_{L^\infty}^2\|\nabla b\|_{L^2}^2\right)\notag\\
	\leq& C\left[\|\sqrt{\rho}u_t\|_{L^2}^2+\|u\|_{L^6}^{\f32}\|\nabla^2 u\|_{L^2}^{\f12}\|\nabla u\|_{L^2}^2+\|b\|_{L^6}^{\f32}\|\nabla^2 b\|_{L^2}^{\f12}\|\nabla b\|_{L^2}^2\right]\\
	\leq& C\left[\|\sqrt{\rho}u_t\|_{L^2}^2+\|\nabla u\|_{L^2}^{\f{14}3}+\|\nabla b\|_{L^2}^{\f{14}3}+\| \nabla^2 b\|_{L^2}^2\right]+\f12\|\nabla^2u\|_{L^2}^2,\notag
\end{align}
which together with \eqref{3.2-1}, \eqref{3.2-01}, \eqref{3.3-1} and \eqref{3.3-16} yields
\begin{align}\label{3.3-17}
	t^2\|\nabla^2 u\|^2_{L^2}+\int_0^Tt\|\nabla^2 u\|^2_{L^2}\,dt\leq C.
\end{align}
Similarly with \eqref{3.1-7}, we first have
\begin{align}\label{3.1-71}
	&\bar{\rho}\int|u_t|^2\,dx=\int\rho|u_t|^2\,dx-\int\left(\rho-\bar{\rho}\right)|u_t|^2\,dx\nonumber\\
	&\leq C\|\sqrt{\rho}u_t\|^2_{L^2}+\|\rho-\bar{\rho}\|_{L^{\frac{3}{2}}}\|u_t\|^2_{L^6}
	\leq C\left(\|\sqrt{\rho}u_t\|^2_{L^2}+\|\nabla u_t\|^2_{L^2}\right),
\end{align}
which together with \eqref{3.2-01} and \eqref{3.3-1} implies \eqref{3.3-01}. Thus, the proof of Lemma \ref{lem.B3.3} is finished.
\end{proof}

\vskip .1in
Based on Lemmas \ref{lem.B3.1}-\ref{lem.B3.3}, we give the last a priori estimates, which is essential (the $L^{1}\left([0, +\infty); L^\infty\right)$ norm of $\nabla u$) to extend the local strong solution to be a global one.
\vskip .1in
\begin{lemma}\label{lem.B3.4}
Supposing $(\rho, u, b)$ is an axisymmetric solution to the system \eqref{1-1}-\eqref{1-2} and \eqref{1-2boundary}, then there exists a genuine constant $C$ depending only on $\bar{\rho}$, $\|\rho_0-\bar{\rho}\|_{L^{\f32}}$,  $\|\nabla {\rho}_0\|_{L^2}$, $\|u_0\|_{H^1}$ and $\|b_0\|_{H^1}$ such that 
\begin{align}\label{3.4-1}
\sup\limits_{t\in[0, T]}\|\nabla\rho\|_{L^2}+\int_0^T\left(\|\nabla u\|_{L^\infty}+\|\rho_t\|^4_{L^{2}}\right)+t^2\left(\|\nabla^2u\|_{L^6}^2+\|\nabla^2b\|_{L^6}^2\right)\,dt\leq C.
\end{align}
\end{lemma}
\begin{proof}
Taking the $x_i-$derivative on the  equation $\eqref{1-1}_1$, multiplying the resultant by $\partial_i\rho$ and adding them up, we have
\begin{align}\label{3.4-3}
\frac{d}{dt}\|\nabla\rho\|_{L^{2}}\leq&C\|\nabla u\|_{L^\infty}\|\nabla\rho\|_{L^{2}},
\end{align}
which implies, after using Gr\"{o}nwall inequality, that
\begin{align}\label{3.4-4}
\|\nabla\rho\|_{L^{2}}\leq \|\nabla\rho_0\|_{L^{2}}\exp\left\{C\int_0^T\|\nabla u\|_{L^\infty}\,dt\right\}.
\end{align}
To bound \eqref{3.4-4}, it suffices to estimate $\int_0^T\|\nabla u\|_{L^\infty}\,dt$. As a preparation, we first establish the estimates of $\|\nabla^2u\|_{L^p}$ for $p\in[2,6]$ and $\|\nabla^2u\|_{L^6}$ in different ways. According to 
Lemma \ref{lem.B5}, H\"{o}lder inequality
Lemma \ref{newgn} and Young inequality, for any $p\in[2,6]$, it yields that 
\begin{align}\label{3.4-5}
	&\|\nabla^2u\|_{L^p}+\|\nabla p\|_{L^p}\nonumber\\
	\leq&C\left(\|\rho u_t\|_{L^p}+\|\rho u\cdot\nabla u\|_{L^p}+\|b\cdot\nabla b\|_{L^p}\right)\nonumber\\
	\leq&C\left(\|\sqrt{\rho}u_t\|_{L^2}^{\frac{6-p}{2p}}\| u_t\|_{L^6}^{\frac{3p-6}{2p}}+\|u\|_{L^\infty}\|\nabla u\|_{L^p}+\|b\|_{L^\infty}\|\nabla b\|_{L^p}\right)\\
	\leq&C\|\sqrt{\rho}u_t\|_{L^2}^{\frac{6-p}{2p}}\|\nabla u_t\|_{L^2}^{\frac{3p-6}{2p}}+C\| u\|^{\f34}_{L^6}\|\nabla^2u\|^\f14_{L^2}\|\nabla u\|^{\f2p}_{L^2}\|\nabla^2u\|^\f{p-2}{p}_{L^2}\nonumber\\
	&+C\| b\|_{L^6}^{\f34}\|\nabla^2b\|^\f14_{L^2}\|\nabla b\|^{\f2p}_{L^2}\|\nabla^2b\|^\f{p-2}{p}_{L^2}\nonumber\\
	\leq&C\|\sqrt{\rho}u_t\|_{L^2}^{\frac{6-p}{2p}}\|\nabla u_t\|_{L^2}^{\frac{3p-6}{2p}}+C\|\nabla u\|^{\f{3p+8}{4p}}_{L^2}\|\nabla^2u\|^\f{5p-8}{4p}_{L^2}+C\|\nabla b\|^{\f{3p+8}{4p}}_{L^2}\|\nabla^2b\|^\f{5p-8}{4p}_{L^2}\nonumber\\
	\leq&C\|\sqrt{\rho}u_t\|_{L^2}^{\frac{6-p}{2p}}\|\nabla u_t\|_{L^2}^{\frac{3p-6}{2p}}+C\left(\|\nabla u\|^2_{L^2}+\|\nabla^2u\|^2_{L^2}+\|\nabla b\|^2_{L^2}+\|\nabla^2b\|^2_{L^2}\right).\nonumber
\end{align}
Meanwhile, by using similar tools to deal with \eqref{3.4-5} and classical elliptic theory, we have
\begin{align}\label{3.4-50}
	&\|\nabla^2u\|_{L^6}+\|\nabla p\|_{L^6}+\|\nabla^2b\|_{L^6}\nonumber\\
	\leq&C\left(\|\rho u_t\|_{L^6}+\|\rho u\cdot\nabla u\|_{L^6}+\|b\cdot\nabla b\|_{L^6}+\| b_t\|_{L^6}+\|u\cdot\nabla b\|_{L^6}+\|b\cdot\nabla u\|_{L^6}\right)\nonumber\\
	\leq&C(\bar{\rho})\left(\| u_t\|_{L^6}+\| b_t\|_{L^6}+\|u\|_{L^\infty}\|\nabla u\|_{L^6}+\|b\|_{L^\infty}\|\nabla b\|_{L^6}\right)\\
	&+C\left(\| b_t\|_{L^6}+\|u\|_{L^\infty}\|\nabla b\|_{L^6}+\|b\|_{L^\infty}\|\nabla u\|_{L^6}\right)\nonumber\\
	\leq&C\left(\|\nabla u_t\|_{L^2}+\|u\|^\f56_{L^6}\|\nabla^2 u\|^\f16_{L^6}\|\nabla u\|^\f35_{L^2}\|\nabla^2 u\|^\f25_{L^6}+\|b\|^\f56_{L^6}\|\nabla^2 b\|^\f16_{L^6}\|\nabla b\|^\f35_{L^2}\|\nabla^2 b\|^\f25_{L^2}\right)\nonumber\\
	&+C\left(\|\nabla b_t\|_{L^2}+\|u\|^\f56_{L^6}\|\nabla^2 u\|^\f16_{L^6}\|\nabla b\|^\f35_{L^2}\|\nabla^2 b\|^\f25_{L^6}+\|b\|^\f56_{L^6}\|\nabla^2 b\|^\f16_{L^6}\|\nabla u\|^\f35_{L^2}\|\nabla^2 u\|^\f25_{L^2}\right)\nonumber\\
	\leq&\f12\left(\|\nabla^2 u\|^2_{L^6}+\|\nabla^2 b\|^2_{L^6}\right)+C\left(\|\nabla u_t\|_{L^2}+\|\nabla b_t\|_{L^2}+\|\nabla u\|^\f{43}{13}_{L^2}+\|\nabla b\|^\f{43}{13}_{L^2}\right),\nonumber
\end{align}
which implies, after employing \eqref{3.2-1}, \eqref{3.2-01} and \eqref{3.3-1} that
\begin{align}\label{3.4-500}
	&\|\nabla^2u\|_{L^6}+\|\nabla p\|_{L^6}+\|\nabla^2b\|_{L^6}\nonumber\\
	\leq&C\left(\|\nabla u_t\|_{L^2}+\|\nabla b_t\|_{L^2}+\|\nabla u\|^\f{43}{13}_{L^2}+\|\nabla b\|^\f{43}{13}_{L^2}\right),
\end{align}
and
\begin{align}\label{3.4-5000}
	\int_0^Tt^2\left(\|\nabla^2u\|_{L^6}^2+\|\nabla^2b\|_{L^6}^2\right)\,dt\leq C.
\end{align}
Now, we estimate $\int_0^T\|\nabla u\|_{L^\infty}\,dt$, to this end, by Lemma \ref{newgn}, \eqref{3.4-5} and Young inequality, for $r\in[2, p)$, we first get
\begin{align}\label{3.4-0.5}
	&\|\nabla u\|_{L^\infty}\leq C\|\nabla u\|^\f{r}{2r-2}_{L^2}\|\nabla^2u\|^\f{r-2}{2r-2}_{L^r}\leq C\|\nabla u\|_{L^2}+C\|\nabla^2 u\|_{L^r}\\
	\leq&C\left(\|\sqrt{\rho}u_t\|_{L^2}^{\frac{6-r}{2r}}\|\nabla u_t\|_{L^2}^{\frac{3r-6}{2r}}+\|\nabla u\|^2_{L^2}+\|\nabla^2u\|^2_{L^2}+\|\nabla b\|^2_{L^2}+\|\nabla^2b\|^2_{L^2}+\|\nabla u\|_{L^2}\right).\nonumber
\end{align}
Based on \eqref{3.4-0.5}, for $\zeta(t)\triangleq\min\{1, t\}$ with $t\in[0, T]$, according to \eqref{3.1-2}, \eqref{3.2-1} and \eqref{3.3-1} with $i=1$, it follows that 
\begin{align}\label{3.4-05}
	&\int_0^{\zeta(t)}\|\nabla u\|_{L^\infty}\,dt
	\leq C\sup\limits_{t\in[0, T]}\left(t\|\sqrt{\rho}u_t\|_{L^2}^2\right)^{\frac{6-r}{4r}}\left(\int_0^Tt\|\nabla u_t\|^2_{L^2}\,dt\right)^{\frac{3r-6}{4r}}\left(\int_0^1t^{-{\frac{2r}{r+6}}}\,dt\right)^{\frac{r+6}{4r}}\notag\\
	&+\int_0^T\left(\|\nabla u\|^2_{L^2}+\|\nabla^2u\|^2_{L^2}+\|\nabla^2b\|^2_{L^2}+\|\nabla b\|^2_{L^2}\right)dt
	+\left(\int_0^T\|\nabla u\|_{L^2}^2\,dt\right)^{\f12}\left(\int_0^1dt\right)^{\f12}\notag\\
	&\leq C,
\end{align}
where we have used the fact $r<6$ in the first inequality. 

It suffices to estimate $\int_{\zeta(t)}^T\|\nabla u\|_{L^\infty}\,dt$. Thanks to 
Lemma \ref{newgn}, \eqref{3.2-1}  and \eqref{3.4-500}, there holds that
\begin{align}\label{3.4-05.}
	&\|\nabla u\|_{L^\infty}\leq C\|u\|^\f14_{L^2}\|\nabla^2 u\|^\f34_{L^6}\nonumber\\
	\leq&C\| u\|^\f14_{L^2}\left(\|\nabla u_t\|_{L^2}+\|\nabla b_t\|_{L^2}+\|\nabla u\|^\f{43}{13}_{L^2}+\|\nabla b\|^\f{43}{13}_{L^2}\right)^\f34\\
	\leq&C\left(\|\nabla u_t\|^\f34_{L^2}+\|\nabla b_t\|^\f34_{L^2}+\|\nabla u\|^2_{L^2}+\|\nabla b\|^2_{L^2}\right),\nonumber
\end{align}
which implies, after applying
\eqref{3.1-2}, \eqref{3.2-1}and \eqref{3.3-1} with $i=2$, that
\begin{align}\label{3.4-6}
	&\int_{\zeta(t)}^T\|\nabla u\|_{L^\infty}\,dt\nonumber\\
	\leq&\int_{\zeta(t)}^T\left(\|\nabla u_t\|^\f34_{L^2}+\|\nabla b_t\|^\f34_{L^2}+\|\nabla u\|^2_{L^2}+\|\nabla b\|^2_{L^2}\right)\,dt\\
	\leq&C+\left(\int_0^Tt^2\|\nabla u_t\|^2_{L^2}\,dt\right)^\f38\left(\int_1^Tt^{-\f65}\,dt\right)^\f58+\left(\int_0^Tt^2\|\nabla b_t\|^2_{L^2}\,dt\right)^\f38\left(\int_1^Tt^{-\f65}\,dt\right)^\f58
	\nonumber\\
	\leq&C+C(1-T^{-\f15})^\f58\leq C.\nonumber
\end{align}
Summing up \eqref{3.4-05} and \eqref{3.4-6}, we finally obtain
\begin{align}\label{3.4-7}
\int_0^{T}\|\nabla u\|_{L^\infty}\,dt\leq C.
\end{align}
Substituting \eqref{3.4-7} into \eqref{3.4-4}, we have
\begin{align}\label{3.4-8}
\sup\limits_{t\in[0, T]}\|\nabla\rho\|_{L^{2}}\leq C\|\nabla\rho_0\|_{L^{2}}.
\end{align}
According to the equation $\eqref{1-1}_1$,  H\"{o}lder inequality, Lemma \ref{newgn}, \eqref{3.2-1} and \eqref{3.4-8}, we have
\begin{align}\label{3.4-9}
	\|\rho_t\|^4_{L^{2}}&\leq\|u\cdot\nabla\rho\|^4_{L^{2}}\leq\|u\|^4_{L^{\infty}}\|\nabla\rho\|^4_{L^{2}}\leq C\|u\|^2_{L^2}\|\nabla^2u\|^2_{L^{2}}\|\nabla\rho\|^4_{L^{2}}\\
	&\leq C\|\nabla^2 u\|^2_{L^2},\nonumber
\end{align}
which completes the proof of \eqref{3.4-1} after integrating \eqref{3.4-9} and using \eqref{3.2-1} again.
\end{proof}

\subsection{Proof of Theorem \ref{theorem.1}} Thanks to Theorem \ref{theorem.1.0}, there exists a time interval $T_{\ast}>0$ such that the system \eqref{1-1}-\eqref{1-2} has a unique local strong solution $(\rho, u, b)$ on $[0, T_{\ast}]\times\Omega$.

Subsequently, we intend to extend the aforesaid local solution to be a global one. To this end, we define
\begin{align}\label{4-1}
	T^{\ast}=\sup\{\,\,T\in{\R}^+\,\,|\,(\rho, u, b)\,\,\mathrm{is}\,\,\mathrm{a}\,\,\mathrm{strong}\,\,\mathrm{solution}\,\,\mathrm{on}\,\,(0, T]\times\Omega \},
\end{align}
and clearly $T^{\ast}>0$. For any $0<\tau<T\leq T^{\ast}$ with $T^{\ast}$ be finite, according to \eqref{3.2-1} and \eqref{3.3-1}, it follows that
\begin{align}\label{4-3}
	\left(\nabla u, \nabla b\right)\in C([\tau, T]; L^2),
\end{align}
where we have used the Sobolev embedding
\begin{align}
	\left(\nabla u,\,\nabla b\right)\in L^\infty([0, T]; H^1)\cap H^1([\tau, T]; L^2)\hookrightarrow C([\tau, T]; L^2).
\end{align}
Utilizing \eqref{3.1-2}, \eqref{3.1-7}, \eqref{3.3-01}, \eqref{3.2-1} and the Sobolev embedding, there holds that
\begin{align}\label{4-4.0}
	u\in H^{1}([\tau, T]; L^2)\hookrightarrow C([\tau, T]; L^2),
\end{align}
and
\begin{align}\label{4-4.1}
	 b\in H^{1}([0, T]; L^2)\hookrightarrow C([0, T]; L^2).
\end{align}
With the help of \eqref{3.1-1}, \eqref{3.4-1} and \eqref{3.4-4}, it yields that
\begin{align}\label{4-4}
\rho-\bar{\rho}\in C([0, T]; L^\frac{3}{2}\cap L^\infty\cap \dot{H}^1).
\end{align}

Now, we claim that
\begin{align}\label{4-5}
	T^{\ast}=\infty,
\end{align}
otherwise, if $T^{\ast}<\infty$, it follows from \eqref{4-3}, \eqref{4-4.0}, \eqref{4-4.1}, \eqref{4-4} that
\begin{align}
\left(\rho^\ast, u^\ast, b^\ast\right)(x, T^{\ast})=\lim\limits_{t\rightarrow T^{\ast}}(\rho, u, b)(x, t),
\end{align}
and
\begin{align}
\rho^\ast-\bar{\rho}\in L^\frac{3}{2}\cap L^\infty\cap \dot{H}^1,\quad u^\ast\in H^1_{0, \sigma},\quad b^\ast\in H^1_{0, \sigma}.
\end{align}
In consequence, we can take $\left(\rho^\ast, \rho^\ast u^\ast, b^\ast\right)$ as the new initial data and apply Theorem \ref{theorem.1.0} to extend the maximal existence time of local strong solution beyond $T^\ast$. This contradicts the hypothesis of $T^{\ast}$ in \eqref{4-1}, therefore \eqref{4-5} holds. Besides, \eqref{1-9} and \eqref{1-10} follow from Lemmas \ref{lem.B3.01}-\ref{lem.B3.4} directly,
thus we finish the proof of Theorem \ref{theorem.1}.

\section{Local well-posedness of strong solutions}

For the integrity of current paper, in this section, we present the proof of local existence and uniqueness of strong solutions (i.e. Theorem \ref{theorem.1.0}) to the system \eqref{1-1}-\eqref{1-2} without any compatibility condition on the initial data. To this end, we borrow the idea developed by the paper \cite{b1.12} and take some modifications. For keeping the presentation terse, we only list the main steps.
\begin{theorem}\label{theorem.1.0}
Let $\Omega$ be the exterior of a cylinder, the initial data $(\rho_0, u_0, b_0)$ is axisymmetric and satisfies
\begin{align}\label{1-7.0}
0\leq\rho_0\leq\bar{\rho},\,\,\rho_0-\bar{\rho}\in L^{\frac{3}{2}}\cap\dot{H}^1(\Omega),\,\,u_0\in H^1_{0, \sigma}(\Omega),\,\,b_0\in H^1_{0, \sigma}(\Omega),
\end{align}
for some $\bar{\rho}>0$. Then there exists a finite time $T_0>0$ such that the system \eqref{1-1}-\eqref{1-2} and \eqref{1-2boundary} has a unique strong solution $(\rho, u, b)$ on $[0, T_0]\times\Omega$ so that for any $2\leq q<\infty$,
\begin{align}\label{1-7.1}
\left\{
\begin{aligned}
&0\leq\rho-\bar{\rho}\in L^{\infty}([0, T_0]; L^{\f32}\cap L^{\infty}\cap \dot{H}^1(\Omega))\cap C([0, T_0];L^{q}(\Omega)),\\
&\rho u\in C([0, T_0]; L^2(\Omega)),\,\,\rho_t\in L^4([0, T_0]; L^2(\Omega)),\,\,\sqrt{\rho}u_t\in L^2([0, T_0]; L^2(\Omega)),\\
&u\in L^\infty([0, T_0]; H^1_{0,\sigma}(\Omega))\cap L^2([0, T_0]; {H}^2(\Omega)),\,\,\sqrt{t}u_t\in L^2([0, T_0]; H^1(\Omega)),\\
&\sqrt{t} \nabla u\in L^\infty([0, T_0]; L^2(\Omega))\cap L^2([0, T_0]; \dot{H}^1(\Omega)),\\
&{t} \nabla^2u,{t} \nabla^2b\in L^\infty([0, T_0]; L^2(\Omega))\cap L^2([0, T_0]; L^6(\Omega)),\\
&b\in L^\infty([0, T_0]; L^4\cap H^1_{0,\sigma}(\Omega))\cap L^2([0, T_0]; {H}^2),\,\,b_t\in L^2([0, T_0]; L^2(\Omega)),\\
&\sqrt{t} \nabla b\in L^\infty([0, T_0]; H^1(\Omega))\cap L^2([0, T_0]; \dot{H}^1(\Omega)),\,\,\sqrt{t}b_t\in L^2([0, T_0]; H^1(\Omega)).
\end{aligned}
\right.
\end{align}
\end{theorem}
\begin{proof}
\textbf{Step1. Construction of approximated solutions}\\

Initially, we regularize the initial data $(\rho_0, u_0, b_0)$ via the standard mollifying process. Let
$\Omega_R\triangleq\Omega\cap\{|x|<R\}$ with $R\gg 1$, $\rho_{0, R}\triangleq(\rho_0)_R+R^{-1}e^{-|x|^2}$ and $(\rho_0)_R\in C^{\infty}(\Omega_R)$ such that
\begin{align*}
0\leq(\rho_0)_R\leq\bar{\rho},\,\,(\rho_0)_R-\bar{\rho}\rightarrow\rho_0-\bar{\rho}\,\, \mathrm{in}\,\,L^{\frac{3}{2}}(\Omega_R)\cap \dot{H}^{1}(\Omega_R)\cap L^\infty(\Omega_R),\,\,\mathrm{as}\,\, R\rightarrow\infty,
\end{align*}
and therefore
\begin{align}\label{rho0R}
	\rho_{0, R}-\bar{\rho}\rightarrow\rho_0-\bar{\rho}\,\, \mathrm{in}\,\,L^{\frac{3}{2}}(\Omega_R)\cap \dot{H}^{1}(\Omega_R)\cap L^\infty(\Omega_R),\,\,\mathrm{as}\,\, R\rightarrow\infty.
\end{align}
Setting $\tilde{u}_{0, R}\in H^1_{0, \sigma}(\Omega_R)$ to be the solution of
\begin{align*}
-\triangle\tilde{u}_{0, R}+\nabla p_{0, R}=-\triangle u_0\,\,\mathrm{in}\,\,\Omega_R,
\end{align*}
and extending $\tilde{u}_{0, R}$ to $\mathbb{R}^3$ by defining $0$ outside $\Omega_R$ so that
\begin{align*}
\tilde{u}_{0, R}\rightarrow u_0\,\,\mathrm{in}\,\,H^1_{0, \sigma}(\Omega_R),\,\,\mathrm{as}\,\,R\rightarrow\infty.
\end{align*}
Then defining $u_{0, R}\triangleq\tilde{u}_{0, R}\ast\omega_{R^{-1}}\in C^{\infty}_{0, \sigma}(\Omega_R)\cap C^{\infty}_{0, \sigma}(\Omega)$,  where $\omega_{R^{-1}}$ is the standard Friedrich mollifier with width $R^{-1}$ and hence
\begin{align}\label{u0R}
u_{0, R}\rightarrow u_0\,\,\mathrm{in}\,\,H^1_{0, \sigma}(\Omega_R),\,\,\mathrm{as}\,\,R\rightarrow\infty.
\end{align}
Recalling that $b_0\in H^1_{0, \sigma}(\Omega)$, we can choose $b_{0, R}\in\{b\in C_0^\infty(\Omega_R\cap\Omega)\,|\,\mathrm{div}\,b=0\}$ such that
\begin{align*}
b_{0, R}\rightarrow b_0\,\,\mathrm{in}\,\,H^1_{0, \sigma}(\Omega_R),\,\,\mathrm{as}\,\,R\rightarrow\infty.
\end{align*}
\textbf{Step2. Passing to the limit}\\

With the help of a priori estimates established in section 3, it is clear that there exists a $T_0$ independent of $R$ such that the system \eqref{1-1}-\eqref{1-2} with initial data $(\rho_{0, R}, u_{0, R}, b_{0, R})$ has a unique smooth solution $(\rho_R, u_R, b_R)$ on $[0, T_0]\times\Omega_R$ and we then extend this solution by $0$ on $\Omega\setminus\Omega_R$.  Thanks to Lemma \ref{lem.B1.0}, the smooth solution  $(\rho_R, u_R, b_R)$ is still axisymmetric.

Letting $R\rightarrow\infty$, according to Lemmas \ref{lem.B3.01}-\ref{lem.B3.4} and because these estimates are independent of the size of $\Omega_R$, there exists an extraction of subsequence of $(\rho_R, u_R, b_R)$ converges to the limit $(\rho, u, b)$ in the weak sense. In particular, for any $\tau>0$ and compact subdomain $\Omega'$, it holds that
\begin{align*}
&\left(\rho_R-\bar{\rho}\right)\rightharpoonup\left(\rho-\bar{\rho}\right)\,\,\mathrm{weakly}\ast\,\,\mathrm{in}\,\,L^{\infty}([0,T_0]; H^{1}\cap L^\infty(\Omega')),\\
&u_R\rightharpoonup u\,\,\mathrm{weakly}\ast\,\,\mathrm{in}\,\,L^{\infty}([0,T_0]; H^1(\Omega'))\cap L^{\infty}([\tau,T_0]; H^2(\Omega')),\\
&u_R\rightharpoonup u\,\,\mathrm{in}\,\,L^2([0,T_0]; H^2(\Omega'))\cap L^2([\tau,T_0]; W^{2,6}(\Omega')),\\
&(u_R)_t\rightharpoonup u_t\,\,\mathrm{in}\,\,L^2([\tau,T_0]; H^1(\Omega')),\,\,(\rho_R)_t\rightharpoonup \rho_t\,\,\mathrm{in}\,\,L^4([0,T_0]; L^{2}(\Omega')),\\
&b_R\rightharpoonup b\,\,\mathrm{weakly}\ast\,\,\mathrm{in}\,\,L^{\infty}([0,T_0]; H^1(\Omega'))\cap L^{\infty}([\tau,T_0]; H^2(\Omega')),\\
&b_R\rightharpoonup b\,\,\mathrm{in}\,\, L^2([0,T_0]; H^2(\Omega'))\cap L^2([\tau,T_0]; W^{2,6}(\Omega')),\\
&(b_R)_t\rightharpoonup b_t\,\,\mathrm{in}\,\,L^{2}([\tau,T_0]; H^1(\Omega')),
\end{align*}
which combine with the Aubin-Lions compactness lemma further implies
\begin{align*}
&u_R\rightarrow u\,\,\mathrm{in}\,\,C([\tau, T_0]; H^1\cap L^6(\Omega'))\cap L^2([\tau, T_0]; C^1(\overline{\Omega'}))\cap L^2([0, T_0]; C^{0,\alpha}(\overline{\Omega'})),\\
&\left(\rho_R-\bar{\rho}\right)\rightarrow \left(\rho-\bar{\rho}\right)\,\,\mathrm{in}\,\,C([0, T_0]; L^q(\Omega')),\quad{\rm for}\,\,{\rm any}\,\,2\leq q<\infty,\\
&b_R\rightarrow b\,\,\mathrm{in}\,\,C([\tau, T_0];H^1\cap L^6(\Omega'))\cap L^2([0, T_0]; C^{0,\alpha}(\overline{\Omega'})),\quad{\rm for}\,\,{\rm any}\,\,0<\alpha<\f12.
\end{align*}
Due to the previous convergences, when $R\rightarrow\infty$, we can derive that
\begin{align*}
&\rho_R(u_R)_t\rightharpoonup\rho u_t\,\,\mathrm{in}\,\,L^2([\tau, T_0]; L^2(\Omega')),\,\,u_R\cdot\nabla\rho_R\rightharpoonup u\cdot\nabla\rho\,\,\mathrm{in}\,\,L^2([\tau, T_0]; L^{2}(\Omega')),\\	
&\rho_Ru_R\cdot\nabla u_R\rightharpoonup \rho u\cdot\nabla u\,\,\mathrm{in}\,\,L^2([\tau, T_0]; L^{2}(\Omega')),\,\,b_R\cdot\nabla b_R\rightharpoonup b\cdot\nabla b\,\,\mathrm{in}\,\,L^2([0, T_0]; L^{2}(\Omega')),\\
&u_R\cdot\nabla b_R\rightharpoonup u\cdot\nabla b\,\,\mathrm{in}\,\,L^2([0, T_0]; L^{2}(\Omega')),\,\,b_R\cdot\nabla u_R\rightharpoonup b\cdot\nabla u\,\,\mathrm{in}\,\,L^2([0, T_0]; L^{2}(\Omega')),
\end{align*}
for any $\tau\in(0, T_0)$ and compact subdomain $\Omega'$. Hence, $(\rho, u, b)$ satisfies the system
\eqref{1-1}-\eqref{1-2} in the sense of distribution and further a.e. in $\Omega'\times(0, T_0)$ by
the regularities stated in Theorem \ref{theorem.1.0}. It remains to verify $\rho u\in C([0, T_0];  L^2(\Omega'))$. Firstly, for any $t\in(0, T_0)$, it follows from Gagliardo-Nirenberg inequality, H\"{o}lder inequality and \eqref{3.2-1} that
\begin{align}\label{4-107}
&\|(\rho_R u_R)(t)-\rho_{0,R}u_{0,R}\|_{L^1}
=\|\int_0^t(\rho_R u_R)_t\,d\tau\|_{L^1}
=\|\int_0^t(\partial_t\rho_R u_R+\rho_R\partial_t u_R)\,d\tau\|_{L^1}\nonumber\\
\leq&\int_0^t(\|\partial_t\rho_R u_R\|_{L^1}+\|\rho_R\partial_t u_R\|_{L^1})\,d\tau
\leq C\int_0^t(\|\partial_t\rho_R\|_{L^{\f32}} \|u_R\|_{L^\infty}+\sqrt{\bar{\rho}}\|\sqrt{\rho_R}\partial_t u_R\|_{L^2})\,d\tau\nonumber\\
\leq&C\int_0^t(\|\partial_t\rho_R\|_{L^{\f32}}\|\nabla u_R\|^{\frac{1}{2}}_{L^2}\|\nabla u_R\|^{\frac{1}{2}}_{H^1}+\|\sqrt{\rho_R}\partial_t u_R\|_{L^2})\,d\tau\\
\leq&C\sqrt{t}\left[\left(\int_0^t\|\partial_t\rho_R\|_{L^{\f32}}^4\,d\tau\right)^{\frac{1}{4}}\left(\int_0^t\|\nabla u_R\|^2_{H^1}\,d\tau\right)^{\frac{1}{4}}+\left(\int_0^t\|\sqrt{\rho_R}\partial_t u_R\|^2_{L^2}\,d\tau\right)^{\frac{1}{4}}\right]\nonumber\\
\leq&C\sqrt{t},\nonumber
\end{align}
which yields, after applying \eqref{3.2-1} and \eqref{u0R},
\begin{align}\label{4-108}
&\|(\rho_R u_R)(t)-\rho_{0,R}u_{0,R}\|_{L^2}
\leq\|(\rho_Ru_R)(t)-\rho_{0,R}u_{0,R}\|^{\frac{2}{5}}_{L^1}\|(\rho_Ru_R)(t)-\rho_{0,R}u_{0,R}\|^{\frac{3}{5}}_{L^6}\notag\\
\leq& Ct^{\frac{1}{5}}\left(\|\nabla u_R\|_{L^2}(t)+\|\nabla u_{0,R}\|_{L^2}\right)^{\f35},
\end{align}
Hence, it deduces that
\begin{align}\label{4-109}
&\|(\rho u)(t)-\rho_{0,R}u_{0,R}\|_{L^2}\leq\|(\rho u)(t)-(\rho_Ru_R)(t)\|_{L^2}+\|(\rho_Ru_R)(t)-\rho_{0,R}u_{0,R}\|_{L^2}\nonumber\\
&+\|\rho_{0,R}u_{0,R}-\rho_{0,R}u_0\|_{L^2}+\|\rho_{0,R}u_{0}-\rho_{0}u_0\|_{L^2}\\
\leq&\|(\rho u)(t)-(\rho_Ru_R)(t)\|_{L^2}+Ct^{\frac{1}{5}}+\|\rho_{0,R}u_{0,R}-\rho_{0,R}u_0\|_{L^2}+\|\rho_{0,R}u_{0}-\rho_{0}u_0\|_{L^2}.\nonumber
\end{align}
On one hand, thanks to $\left(\rho_R-\bar{\rho}\right)\rightarrow \left(\rho-\bar{\rho}\right)$ in $C([0, T_0]; L^q)$ for any $2\leq q<\infty$ and $u_R\rightarrow u$ in $C([\tau, T_0]; H^1\cap L^6)$, we have $\rho_Ru_R\rightarrow\rho u\,\mathrm{in}\,C([\tau, T_0]; L^2)$. On the other hand, due to \eqref{rho0R}, \eqref{u0R}, $\left(\rho_R-\bar{\rho}\right)\rightarrow \left(\rho-\bar{\rho}\right)$ in $C([0, T_0]; L^q)$ for any $2\leq q<\infty$ and $u_0\in H^1_{0, \sigma}$, we can deduce that the last two terms on the right side of \eqref{4-109} tend to zero. In summary, we obtain
\begin{align}\label{4-110}
&\|(\rho u)(t)-\rho_0u_0\|_{L^2}\\
\leq&\varliminf\limits_{R\rightarrow\infty}\left(\|(\rho u)(t)-(\rho_R u_R)(t)\|_{L^2}+\|\rho_{0,R}u_{0,R}-\rho_{0,R}u_0\|_{L^2}+\|\rho_{0,R}u_{0}-\rho_{0}u_0\|_{L^2}\right)+Ct^{\frac{1}{5}}\nonumber\\
&=Ct^{\frac{1}{5}},\nonumber
\end{align}
which shows that $\rho u$ is continuous at the original time and satisfies the initial condition $\rho u|_{t=0}=\rho_0u_0$.\\
\textbf{Step 3. The uniqueness of solutions}\\
Assume that $(\rho, u, b)$ and $(\tilde{\rho}, \tilde{u}, \tilde{b})$ are two local strong solutions to the system \eqref{1-1}-\eqref{1-2} with the same initial data satisfying \eqref{1-7.1} and setting
\begin{align*}
Q=\rho-\tilde{\rho},\quad U=u-\tilde{u},\quad B=b-\tilde{b},
\end{align*}
then $(Q, U, B)$ satisfies the following system
\begin{align}\label{4.123}
\left\{
\begin{aligned}
&Q_t+\tilde{u}\cdot\nabla Q+U\cdot\nabla\tilde{\rho}=0,\\
&\rho U_t+\rho u\cdot\nabla U-\triangle U+\nabla(p-\tilde{p})\\
&=-Q(\tilde{u}_t+\tilde{u}\cdot\nabla\tilde{u})-\rho U\cdot\nabla\tilde{u}+b\cdot\nabla B+B\cdot\nabla\tilde{b},\\
&B_t-\triangle B=b\cdot\nabla U+B\cdot\nabla\tilde{u}-u\cdot\nabla B-U\cdot\nabla\tilde{b},\\
&\mathrm{div}\,U=\mathrm{div}\,B=0.
\end{aligned}
\right.
\end{align}
Multiplying $\eqref{4.123}_1$ by $|Q|^{-\frac{1}{2}}Q$ and integrating the resultant by parts, it deduces from H\"{o}lder inequality and Sobolev inequalities that
\begin{align}
\frac{2}{3}\frac{d}{dt}\|Q\|^{\frac{3}{2}}_{L^{\frac{3}{2}}}
\leq\|U\|_{L^6}\|\nabla\tilde{\rho}\|_{L^2}\|Q\|_{L^{\frac{3}{2}}}^{\frac{1}{2}}
\leq C\|\nabla U\|_{L^2}\|\nabla\tilde{\rho}\|_{L^2}\|Q\|_{L^{\frac{3}{2}}}^{\frac{1}{2}},
\end{align}
which yields, after applying \eqref{1-7.1}, that
\begin{align}\label{4.125}
\frac{d}{dt}\|Q\|_{L^{\frac{3}{2}}}\leq C\int_0^t\|\nabla U\|_{L^2}\,d{\tau}.
\end{align}
Taking inner product of $\eqref{4.123}_2$ with $U$ and $\eqref{4.123}_3$ with $B$ respectively and integrating by parts, it follows from \eqref{1-7.1} and H\"{o}lder inequality, Lemma \ref{lem.B2} and Young inequality that
\begin{align}
&\frac{1}{2}\frac{d}{dt}\left(\|\sqrt{\rho}U\|^2_{L^2}+\|B\|^2_{L^2}\right)+\|\nabla U\|^2_{L^2}+\|\nabla B\|^2_{L^2}\nonumber\\
\leq&\int_{\Omega}\Big|\left[\rho U\cdot\nabla\tilde{u}+Q\left(\tilde{u}_t+\tilde{u}\cdot\nabla\tilde{u}\right)\right]\cdot U\Big|\,dx+\int_{\Omega}|\tilde{u}||\nabla B||B|\,dx\nonumber\\
&+\int_{\Omega}|b\cdot\nabla U\cdot B|\,dx+\int_{\Omega}|\nabla\tilde{b}||B||U|\,dx\nonumber\\
\leq&\sqrt{\bar{\rho}}\|\sqrt{\rho}U\|_{L^2}\|U\|_{L^6}\|\nabla\tilde{u}\|_{L^3}
+\|Q\|_{L^{\frac{3}{2}}}\left(\|\tilde{u}_t\|_{L^6}\|U\|_{L^6}+\|\tilde{u}\|_{L^{\infty}}\|\nabla\tilde{u}\|_{L^6}\|U\|_{L^6}\right)\nonumber\\
&+\|\tilde{u}\|_{L^\infty}\|B\|_{L^2}\|\nabla B\|_{L^2}+\|B\|_{L^2}\|\nabla U\|_{L^2}\|b\|_{L^\infty}+\|\nabla\tilde{b}\|_{L^3}\| B\|_{L^2}\|U\|_{L^6}\\
\leq&C\|\sqrt{\rho}U\|_{L^2}\|\nabla U\|_{L^2}\|\nabla\tilde{u}\|_{H^1}+C\|Q\|_{L^{\frac{3}{2}}}\left(\|\nabla\tilde{u}_t\|_{L^2}\|\nabla U\|_{L^2}+\|\nabla\tilde{u}\|^{\frac{1}{2}}_{L^2}\|\nabla\tilde{u}\|^{\frac{3}{2}}_{H^1}\|\nabla U\|_{L^2}\right)\nonumber\\
&+C\|\nabla\tilde{u}\|^{\frac{1}{2}}_{L^2}\|\nabla\tilde{u}\|^{\frac{1}{2}}_{H^1}\|B\|_{L^2}\|\nabla B\|_{L^2}+C\|\nabla b\|^{\frac{1}{2}}_{L^2}\|\nabla b\|^{\frac{1}{2}}_{H^1}\|B\|_{L^2}\|\nabla U\|_{L^2}\nonumber\\
&+C\|\nabla\tilde{b}\|_{H^1}\| B\|_{L^2}\|\nabla U\|_{L^2}\nonumber\\
\leq&\frac{1}{2}\left(\|\nabla U\|^2_{L^2}+\|\nabla B\|^2_{L^2}\right)+C\|Q\|_{L^{\frac{3}{2}}}^2\left(\|\nabla\tilde{u}_t\|^2_{L^2}
+\|\nabla\tilde{u}\|_{L^2}\|\nabla\tilde{u}\|^3_{H^1}\right)\nonumber\\
&+C\left(\|\sqrt{\rho}U\|^2_{L^2}+\|B\|^2_{L^2}\right)\left(\|\nabla\tilde{u}\|^2_{H^1}+\|\nabla b\|^2_{H^1}+\|\nabla\tilde{b}\|_{H^1}^2\right),\nonumber
\end{align}
that is
\begin{align}\label{4.101}
&\frac{d}{dt}\left(\|\sqrt{\rho}U\|^2_{L^2}+\|B\|^2_{L^2}\right)+\left(\|\nabla U\|^2_{L^2}+\|\nabla B\|^2_{L^2}\right)\nonumber\\
\leq&\alpha(t)\left(\|\sqrt{\rho}U\|^2_{L^2}+\|B\|^2_{L^2}\right)+\beta(t)\|Q\|_{L^{\frac{3}{2}}}^2,
\end{align}
where $\alpha(t)\triangleq(\|\nabla\tilde{u}\|^2_{H^1}+\|\nabla b\|^2_{H^1}+\|\nabla\tilde{b}\|_{H^1}^2)(t)$, $\beta(t)\triangleq(\|\nabla\tilde{u}_t\|^2_{L^2}
+\|\nabla\tilde{u}\|_{L^2}\|\nabla\tilde{u}\|^3_{H^1})(t)$. Thanks to \eqref{1-7.1}, we have $\alpha(t),\,t\beta(t)\in L^1(0, T)$. Then by setting $f(t)=\|Q\|_{L^{\frac{3}{2}}}$, $g(t)=\|\sqrt{\rho}U\|^2_{L^2}+\|B\|^2_{L^2}$, $G(t)=\|\nabla U\|^2_{L^2}+\|\nabla B\|^2_{L^2}$
and applying Lemma \ref{lem.B6} to \eqref{4.125} and \eqref{4.101}, we can obtain $\|Q\|_{L^{\frac{3}{2}}}=\|\sqrt{\rho}U\|_{L^2}\equiv\|\nabla U\|_{L^2}\equiv\|\nabla B\|_{L^2}\equiv0$. Thus $Q\equiv U\equiv B\equiv0$, which shows the uniqueness of solutions.
\end{proof}

\section*{Acknowledgments}
Jitao Liu was partially supported by National Natural Science Foundation of China under grant (No. 12471218).

\end{document}